\documentclass[10.8pt]{amsart}

\usepackage{amsmath, amssymb, amsfonts}
\usepackage{graphicx, overpic}
\usepackage{color}
\usepackage{enumerate}
\usepackage{multicol}
\usepackage{xypic}
\usepackage{mathrsfs}
\usepackage[T1]{fontenc}

\usepackage[margin=1.2in]{geometry}

\numberwithin{equation}{section}
\numberwithin{figure}{section}

\theoremstyle{plain}
\newtheorem{thm}{Theorem}[section]
\newtheorem{lemma}[thm]{Lemma}
\newtheorem{prop}[thm]{Proposition}
\newtheorem{coro}[thm]{Corollary}
\theoremstyle{definition}
\newtheorem{defi}[thm]{Definition}
\newtheorem{example}[thm]{Example}
\newtheorem{rema}[thm]{Remark}

\newtheorem{claim}[thm]{Claim}

\newtheorem{nota}[thm]{Notation}
\newtheorem{question}[thm]{Question}

\newcommand{\R}{\mathbb{R}}
\newcommand{\N}{\mathbb{N}}

\newcommand{\Hy}{\mathbb{H}}

\newcommand{\cG}{\mathcal{G}}

\newcommand{\cO}{\mathcal{O}}
\newcommand{\cP}{\mathcal{P}}

\newcommand{\cT}{\mathcal{T}}

\newcommand{\cY}{\mathcal{Y}}

\newcommand{\G}{\Gamma}

\newcommand{\al}{\alpha}
\newcommand{\be}{\beta}
\newcommand{\Si}{\Sigma}

\newcommand{\Isom}{\operatorname{Isom}}
\newcommand{\Aut}{\operatorname{Aut}}

\newcommand{\diam}{\operatorname{diam}}
\newcommand{\Stab}{\operatorname{Stab}}
\newcommand{\Vol}{\operatorname{Vol}}
\newcommand{\Cay}{\operatorname{Cay}}

\newcommand{\p}{\partial}

\newcommand{\wt}{\widetilde}

\newcommand{\fhyp}{\Hy_{\mathbb{F}}^n}

\definecolor{amethyst}{rgb}{0.6, 0.4, 0.8}

\newcommand{\hide}[1]{}

\DeclareMathOperator{\Fix}{Fix}

\title{Action rigidity for free products of hyperbolic manifold groups}
\author{Emily Stark and Daniel J. Woodhouse}
\date{\today}
\address{Department of Mathematics, Wesleyan University. 265 Church Street, Middletown, CT 06457, United States}
\email{estark@wesleyan.edu}
\address{Mathematical Institue, University of Oxford. Andrew Wiles Building, Radcliffe Observatory Quarter, Woodstock Road, Oxford, OX2 6GG, United Kingdom}
\email{daniel.woodhouse@maths.ox.ac.uk}


\begin{document}

\begin{abstract}
 Two groups have a \emph{common model geometry} if they act properly and cocompactly by isometries on the same proper geodesic metric space.
 The Milnor-Schwarz lemma implies that groups with a common model geometry are quasi-isometric; however, the converse is false in general.
 We consider free products of uniform lattices in isometry groups of rank-1 symmetric spaces and prove, within each quasi-isometry class, residually finite groups that have a common model geometry are abstractly commensurable. 
 Our result gives the first examples of hyperbolic groups that are quasi-isometric but do not {\it virtually} have a common model geometry. Indeed, each quasi-isometry class contains infinitely many abstract commensurability classes. 
 We prove that two free products of closed hyperbolic surface groups have a common model geometry if and only if the groups are isomorphic. This result combined with a commensurability classification of Whyte yields the first examples of torsion-free abstractly commensurable hyperbolic groups that do not have a common model geometry. 
 An important component of the proof is a generalization of Leighton's graph covering theorem. 
 The main theorem depends on residual finiteness, and we show that finite extensions of uniform lattices in rank-1 symmetric spaces that are not residually finite would give counterexamples. 
\end{abstract}

\maketitle

    \section{Introduction}

    The study of the large-scale geometry of finitely generated groups seeks to relate three notions: the quasi-isometry class of a group, the abstract commensurability class of a group, and geometric actions of a group on proper geodesic metric spaces. 
    Within this framework, first suggested by Gromov~\cite{gromov}, quasi-isometry and abstract commensurability define equivalence relations on the class of finitely generated groups.
    Moreover, abstract commensurability and geometric actions on a common proper geodesic metric space imply a quasi-isometry (the latter being the Milnor-Schwartz lemma). 
 
    The large-scale geometry of a free product of finitely generated hyperbolic groups depends only on the one-ended factors; the quasi-isometry classification in this setting was given by Papasoglu--Whyte~\cite{papasogluwhyte}.  Martin--\'{S}wiatkowski~\cite{martinswiatkowski} further proved that the boundary of such a group is determined up to homeomorphism by the homeomorphism types of the boundaries of the one-ended factors.
    Thus, there is a great deal of flexibility in creating quasi-isometric groups by free product constructions. In contrast, we prove in this paper that a strong form of rigidity may hold if one requires the groups act geometrically on the same space. 
    
    A {\it model geometry} for a group is a proper geodesic metric space on which the group acts {\it geometrically}, i.e. properly and cocompactly by isometries.
    In parallel to the notion of \emph{quasi-isometric rigidity}, we define a group $G$ to be \emph{action rigid} if any group that shares a common model geometry with $G$ is abstractly commensurable to $G$. 
    For example, closed hyperbolic $n$-manifold groups are not action rigid for each $n \geq 3$, as they all act geometrically on $\Hy^n$, but there are infinitely many abstract commensurability classes of such groups. 
    On the other hand, any group that is quasi-isometrically rigid is action rigid.
    We consider action rigidity within classes of groups for which the quasi-isometry and abstract commensurability classifications do not coincide. 
    
    The first examples of hyperbolic groups that are quasi-isometric but do not have a common model geometry were given by Mosher--Sageev--Whyte~\cite{moshersageevwhyteI}. 
    Let $G_p = \mathbb{Z}/p\mathbb{Z} * \mathbb{Z}/p\mathbb{Z}$ for some prime $p > 2$.
    A group in the class $\{G_p \,| \, p > 2  \textrm{ is prime}\}$ is virtually free and has a natural action on the Bass-Serre tree associated to its splitting as a free product.
    Although all groups in the set $\{G_p \,| \, p > 2  \textrm{ is prime}\}$ are quasi-isometric, Mosher--Sageev--Whyte~\cite{moshersageevwhyteI} showed that the groups $G_p$ and $G_q$ have a common model geometry if and only if $p=q$.
    All groups in this class \emph{virtually} have a common model geometry, meaning that two such groups have finite-index subgroups that have a common model geometry.
    Indeed, any pair of finitely generated, non-abelian free groups act geometrically on the $4$-valent tree.
    The torsion in $G_p$ is precisely the obstruction to finding a common model geometry -- the proof exploits the fact that any proper, minimal action of $G_p$ on a simplicial tree must be the natural action on the $p$-regular tree. 
    A class of groups called {\it simple surface amalgams} gives examples of torsion-free hyperbolic groups that are quasi-isometric but do not have a common model geometry, as shown by the authors~\cite{starkwoodhouse}.
    
    Outside the setting of hyperbolic groups, Das--Tessera~\cite[Theorem 1.1]{DasTessera16} proved that 
    if $\Gamma_g$ denotes the fundamental group of a genus $g \geq 2$ surface, then $\Gamma_g \times \mathbb{Z}$ and the canonical central extension $\wt{\Gamma}_g$ of $\Gamma_g$, are not \emph{integrably measure equivalent}. However, these groups are quasi-isometric. Having a common model geometry implies that two groups are integrably measure equivalent, and integrably measure equivalence is an equivalence relation implied by abstract commensurability. Thus, $\wt{\Gamma}_g$ and $\Gamma_g \times \mathbb{Z}$ are quasi-isometric, but do not virtually have a common model geometry.

    In this paper we give the first examples of hyperbolic groups that are quasi-isometric and do not virtually have a common model geometry. We study action rigidity for free products of closed hyperbolic manifold groups, and, more generally, for the quasi-isometry class of such groups; see Theorem~\ref{thm_FINAL} for a more general statement.
    
    \begin{thm} \label{thm_first} 
     Let $G = H_1 * \ldots * H_k * F_n$, where $H_i$ is a uniform lattice in the isometry group of a rank-1 symmetric space for $1 \leq i \leq k$, and $F_n$ is a finitely generated free group.
     Suppose that $G'$ is residually finite.
     If $G$ and $G'$ have a common model geometry, then $G$ and $G'$ are abstractly commensurable.
    \end{thm}     
   
    Note that in the case where $k=0$, this theorem is just the abstract commensurability of finitely generated free groups. Theorem~\ref{thm_FINAL} is phrased in terms of Stallings-Dunwoody decompositions and we show that after quasi-conjugating the action to a new model geometry, the groups are weakly commensurable in the isometry group the new model geometry. 
    
    Recall, the classification of non-compact rank-$1$ symmetric spaces consists of real hyperbolic space $\mathbb{H}^n = \mathbb{H}^n_{\mathbb{R}}$, complex hyperbolic space $\mathbb{H}^n_{\mathbb{C}}$, quaternionic hyperbolic space $\mathbb{H}^n_{\mathbb{H}}$ (all for $n \geq 2$), and the ``exceptional case'' of the Cayley hyperbolic plane $\mathbb{H}^2_{\textbf{Ca}}$.
    We will use the notation $\mathbb{H}^n_{\mathbb{F}}$ to denote any one of these possible rank-$1$ symmetric spaces, and we define a {\it closed hyperbolic manifold group} to be the fundamental group of a closed manifold that admits the geometry of $\mathbb{H}^n_{\mathbb{F}}$ for some $n$ and $\mathbb{F}$.
    See~\cite{Mostow73} for details. 
    
    It is an open question of considerable interest in the field if hyperbolic groups are residually finite.
    So the residual finiteness assumption in Theorem~\ref{thm_first} could in principal be redundant.
    Conversely, in Section~\ref{sec:hypotheticalCounterexample}, we show that if there exists a non-residually finite finite extension of a uniform lattice in a rank-1 symmetric space, then there exists a pair of groups $G$ and $G'$ that share a common model geometry such that $G$ decomposes as in the statement of the theorem, $G'$ is not residually finite, and $G$ and $G'$ are not even virtually isomorphic. 
    Note that if $G'$ decomposes as a free product of uniform lattices in a  manner similar to $G'$, then $G'$ is residually finite since these lattices are finitely generated and linear, hence residually finite, and free products of residually finite groups are residually finite.
    As a consequence of the resolution of the Virtual Haken Conjecture~\cite{agol, wise_hier}, if $G'$ is cocompactly cubulated, then $G'$ is residually finite.
    If each $H_i$ is quasi-isometric to $\mathbb{H}^2$ or $\mathbb{H}^3$, then the residual finiteness assumption is satisfied. 
    More generally, it is not known if finite extensions of uniform lattices in rank-1 symmetric spaces
    are residually finite.
    Indeed, finite central extensions of lattices in $Sp(n,1)$ are considered by some to be likely candidates for a non-residually finite hyperbolic group.
    In~\cite{GrunewaldJaikin-ZapirainZalesskii08, GrunewaldJaikin-ZapirainPintoZalesskii14} the authors study ``cohomological goodness'', a criterion for residual finiteness to be preserved in finite extensions. See~\cite{Deligne78, Hill19} for non-residually finite examples of (non-hyperbolic) extensions of arithmetic groups. 
    
    Free products of closed hyperbolic manifold groups is a family closed under passing to finite-index subgroups. Moreover, each quasi-isometry class contains infinitely many commensurability classes; see Lemma~\ref{lem:infinitelyManyClasses}. Thus, we have the following.
    
    \begin{coro} \label{cor_qi_no_mg} There are torsion-free hyperbolic groups which are quasi-isometric but cannot \emph{virtually} act on a common model geometry. Moreover, there are examples $G$ and $G'$ for which the ratios of the non-vanishing $\ell^2$-Betti numbers are equal: $\frac{b_k^{(2)}(G)}{b_k^{(2)}(G')}= C$. 
    \end{coro}
    
    \begin{rema} \label{rem:proportionality}
     After the initial preprint for this paper was presented, Kevin Schreve observed that Corollary~\ref{cor_qi_no_mg} can also be deduced for certain examples among the groups considered here via an application of the {\it proportionality principal} for $\ell^2$-Betti numbers due to Gaboriau~\cite{Gaboriau02, Gaboriau2019}.
     This principal states that if $G$ and $G'$ are uniform lattices in a locally compact, second countable group $H$, then their $\ell^2$-Betti numbers are related as follows:
     \[
      \frac{b_k^{(2)}(G)}{\Vol(G \backslash H)} = \frac{b_k^{(2)}(G')}{\Vol(G' \backslash H)},
     \]
     where volume in the formula is given by the Haar measure on $H$.
     If $G$ and $G'$ have a common model geometry $X$, then they embed as uniform lattices in the isometry group $H = \Isom(X)$, which is locally compact and second countable 
     (see~\cite[Lemma 5.B.4]{CornulierDeLaHarpe}). 
     Thus, the ratio $\frac{b_k^{(2)}(G)}{b_k^{(2)}(G')}$ of all the non-vanishing $\ell^2$-Betti numbers of $G$ and $G'$ must be equal.
     Moreover, since both $\ell^2$-Betti numbers and covolume scale by degree upon passing to finite-index subgroups, the ratio is preserved for finite-index subgroups.
     Thus, if some ratios are not equal for $G$ and $G'$, then the corresponding ratios will not be equal for any pair of finite-index subgroups $G_0 \leqslant G$ and $G'_0 \leqslant G$.
     
     To construct a pair of quasi-isometric free products with no common model geometry by these methods, let $M$ and $M'$ be closed hyperbolic $4$-manifolds with distinct Euler characteristic.
     The $\ell^2$-Betti numbers of $\pi_1(M)$ and $\pi_1(M')$ vanish aside from the second, which equals the Euler characteristic.
     Let $G = \pi_1 (M) * \mathbb{Z}$ and $G' = \pi_1 (M') * \mathbb{Z}$, which by Mayer-Vietoris will have $b_1^{(2)}(G) = b_1^{(2)}(G')$ and $b_2^{(2)}(G) \neq b_2^{(2)}(G')$.
     In contrast, these methods cannot be applied to free products of surface groups or closed $3$-manifold groups, which have only one non-vanishing $\ell^2$-Betti number.
    \end{rema}

    The strongest form of action rigidity occurs when the groups considered here are surface groups and there are exactly two factors.     
     
    \begin{thm} \label{thm_intro_surface_case} 
     Let $G \cong \pi_1(S_{g_1}) * \pi_1(S_{g_2})$ and $G' \cong \pi_1(S_{h_1}) * \pi_1(S_{h_2})$ be free products of fundamental groups of closed orientable surfaces of genus at least two. The groups $G$ and $G'$ have a common model geometry if and only if the groups $G$ and $G'$ are isomorphic.  
    \end{thm}
        
    Whyte~\cite[Theorem 1.6]{whyte} proved that if $G =\pi_1(S_{g_1}) * \pi_1(S_{g_2})$ and $G' = \pi_1(S_{h_1}) * \pi_1(S_{h_2})$ are free products of fundamental groups of closed orientable surfaces of genus at least two, then $G$ and $G'$ are abstractly commensurable if and only if $\chi(G) = \chi(G')$, which is equivalent to $g_1 + g_2 = h_1 + h_2$.
    For example, $G \cong \pi_1(S_2)*\pi_1(S_4)$ and $G' \cong \pi_1(S_3)*\pi_1(S_3)$ are abstractly commensurable but do not have a common model geometry by Theorem~\ref{thm_intro_surface_case}. 
    Moreover, if $\hat G$ is isomorphic to finite-index subgroups of $G$ and $G'$, then both $G$ and $\hat G$ have a common model geometry, as do $G'$ and $\hat{G}$, but $G$ and $G'$ do not, so the property of having a common model geometry is not transitive on this family of groups. 
     Thus, in combination with Theorem~\ref{thm_intro_surface_case}, we have the following corollary. 
    
    \begin{coro}
     \begin{enumerate}
      \item There are torsion-free abstractly commensurable hyperbolic groups that do not have a common model geometry. 
      \item  The relation of having a common model geometry is not a transitive relation on the class of torsion-free hyperbolic groups. 
      \item For each $n >0$ there exist $n$ free products of closed hyperbolic manifold groups in the same abstract commensurability class that pairwise do not have a common model geometry. 
     \end{enumerate}
    \end{coro}

    \begin{question}
     Is there a commensurability class of hyperbolic groups that contains an infinite subset consisting of groups that do not pairwise share a common model geometry?
    \end{question}
    
    The homeomorphism type of a cover of a closed surface by degree $d$ is determined by $d$, but this fails for higher-dimensional examples. A hyperbolic $3$-manifold may have many non-homeomorphic covers of the same degree; for example, see the discussion by Friedl--Park--Petri--Raimbault--Ray~\cite{FriedlEtAl}.
    Nonetheless, if the free product with amalgamation is of higher-dimensional hyperbolic manifold groups, information can still be deduced.
    
    \begin{thm} \label{thm_intro_amalgamation_case}
     Let $G \cong \pi_1(M_1) * \pi_1(M_2)$ and $G' \cong \pi_1(M_1') * \pi_1(M_2')$ be free products of fundamental groups of closed orientable hyperbolic manifolds. If the groups $G$ and $G'$ have a common model geometry, then, after possibly permuting the factors, the manifolds $M_i$ and $M_i'$ have the same volume.  
    \end{thm}

     The results in this paper provoke the following questions:
    
    \begin{question}
      If $H$ and $H'$ are one-ended residually finite hyperbolic groups, is $H * H'$ action rigid? 
    \end{question}

    \begin{question}\label{ques_common_simplicial}
     If $H$ and $H'$ are one-ended residually finite hyperbolic groups that act geometrically on the same simplicial complex, are $H$ and $H'$ abstractly commensurable? 
    \end{question}

     Both questions are false in general outside of the hyperbolic setting by work of Burger--Mozes~\cite{burgermozes}. The case that $H$ and $H'$ are closed hyperbolic manifold groups is handled in Proposition~\ref{prop:simp_ac}.

    A closed hyperbolic $n$-manifold group is not quasi-isometrically rigid for all $n \geq 3$, although the {\it class} of such groups is quasi-isometrically rigid in the sense that any group quasi-isometric to $\mathbb{H}_{\mathbb{F}}^n$ does in fact act geometrically on $\mathbb{H}_{\mathbb{F}}^n$. Moreover, a closed hyperbolic manifold group is not action rigid. 
    The following corollary states that when one starts taking connect sums of closed hyperbolic $3$-manifolds, the resulting fundamental groups become action rigid.
    
    \begin{coro}
     Let $M$ be a finite non-trivial connected sum of closed hyperbolic $3$-manifolds.
     Then $\pi_1(M)$ is action rigid.
    \end{coro}

    \begin{proof}
    Suppose that $G'$ shares a common model geometry with $\pi_1(M)$.
    The result follows from Theorem~\ref{thm_first} if we can show that $G'$ is residually finite.
    As $G$ has infinitely many ends, so does $G'$, and the one-ended vertex groups in its Stallings-Dunwoody decomposition will be quasi-isometric to $\mathbb{H}^3$ by~\cite{papasogluwhyte}.
     Any group quasi-isometric to $\mathbb{H}^3$ surjects onto a closed hyperbolic $3$-manifold group with finite kernel by~\cite{tukia}.
     Thus, by~\cite{BergeronWise} we know that $G'$ acts geometrically on a proper, hyperbolic CAT(0) cube complex, so by~\cite{agol} we know $G'$ is virtually special and therefore residually finite.
    \end{proof}

    The corollary provokes the following related question:

     \begin{question}
      Is the fundamental group of a compact, non-geometric $3$-manifold action rigid?
     \end{question}

   As explained in the next two subsections, the proof of the theorems above has two main steps, each of independent interest. The first step is geometric; we show that geometric actions of two infinite-ended non-free hyperbolic groups on an arbitrary common model space can be promoted to geometric actions on a model space with more structure. This strategy to prove action rigidity was also employed by Mosher--Sageev--Whyte~\cite{moshersageevwhyteI} for the virtually free groups defined above and by the authors~\cite{starkwoodhouse} for the class of simple surface amalgams. The second step is topological; we prove a generalization of Leighton's graph covering theorem~\cite{Leighton}, following the methods developed by Woodhouse~\cite{Woodhouse18}, and Shepherd and Gardam--Woodhouse~\cite{ShepherdGardamWoodhouse}.    
     
   \subsection{Common simplicial and hyperbolic model geometries}  
    A central theorem we employ to obtain the results above is the following, which can be viewed as a generalization of the work of Mosher--Sageev--Whyte~\cite{moshersageevwhyteI} on virtually free groups. 
    We say that a model geometry $Y$ for $G$ \emph{decomposes as a tree of spaces} if there is a $G$-equivariant map $p: Y \rightarrow T$, where $T$ is a simplicial tree, and the preimage of a vertex $Y_v := p^{-1}(v)$ is a \emph{vertex space}, and the preimage of the interior of an edge decomposes as a product of the \emph{edge space} and open interval $Z_e \times (0,1) := p^{-1}(e^\circ)$.
    
    \begin{thm} \label{intro_nicespace}
     Let $G$  be a hyperbolic group with infinitely many ends, and suppose $G$ is not virtually free. 
     Let $X$ be a model geometry for $G$ and let $H = \Isom(X)$.
     Then, there exists a locally finite, simply connected simplicial complex $Y$ such that: 
     \begin{enumerate}
      \item there is an $H$-action on $Y$;
      \item  $Y$ decomposes as an $H$-equivariant tree of spaces with each edge space a point and each vertex space either one-ended or a point; 
     
      \item there is a quasi-isometry $f: X \rightarrow Y$ that quasi-conjugates the respective $H$ actions. That is to say, there is a constant $B> 0$ such that
     \[
      d_Y(h \cdot f(x), f(h \cdot x)) < B.
     \]
     \end{enumerate}
     As a consequence, $Y$ is a model geometry for $G$, and the one-ended vertex spaces in $Y$ are quasi-isometric to the one-ended vertex groups in the Stallings--Dunwoody decomposition of $G$.
    \end{thm}

    The assumption that the group $G$ is infinite-ended is necessary in general. For example, we show in Proposition~\ref{prop:simp_ac} that, while non-commensurable closed hyperbolic manifold groups have a common model geometry, $\Hy^n_{\mathbb{F}}$, they cannot act geometrically on the same simplicial complex. 
    
    A hyperbolic group admits a {\it Stallings--Dunwoody decomposition} \cite{stallings,dunwoody85} as a finite graph of groups with finite edge groups and vertex groups with at most one end.
    We note that while the graph of groups decomposition is not necessarily unique, the one-ended vertex groups are unique (up to conjugation).
    While a quasi-isometry need not induce an isomorphism from a Bass-Serre tree for a Stallings--Dunwoody decomposition of $G$ to a Bass-Serre tree for a Stallings--Dunwoody decomposition of $G'$, the common model geometry given in Theorem~\ref{intro_nicespace} defines an isomorphism from a Bass-Serre tree for $G$ to a Bass-Serre tree for $G'$. This isomorphism is a crucial component in the proof of action rigidity for free products of hyperbolic manifold groups. 
    
    To prove Theorem~\ref{intro_nicespace} in Section~\ref{sec:common_simp_geo} we use the visual boundary of hyperbolic groups with infinitely many ends. Each conjugate of a one-ended vertex group corresponds to a component in the boundary. We build a locally finite simplicial complex admitting geometric actions by $G$ and $G'$ by considering the set of {\it weak convex hulls} of these components. We take $R$-neighborhoods of these weak convex hulls and define a graph with vertices corresponding to certain intersections of these subsets. 
    We apply the Rips complex construction to this graph to obtain a simply connected model geometry for both $G$ and $G'$. Finally, we use Dunwoody's tracks~\cite{dunwoody85} to collapse this simplicial complex to the desired tree of spaces described in Theorem~\ref{intro_nicespace}. 

    In the case that the one-ended vertex groups of $G$ and $G'$ are closed, real hyperbolic manifold groups, we apply the work of Tukia~\cite{tukia}, Hinkkanen~\cite{hinkkanen85,hinkkanen90} and Markovic~\cite{markovic06} to replace the one-ended vertex spaces in the simplicial complex $Y$ with copies of $\fhyp$, for varying $n>1$. In the complex, quarternionic and Cayley hyperbolic cases we apply corresponding results due to Chow and Pansu~\cite{Chow96, Pansu89}. See Section~\ref{sec:hyp_model_geo}.

    \subsection{Symmetry Restricted Leighton's Theorem} \label{sec:SymmetryRestrictedLeigthon}
    Having promoted the common model geometry to a tree of spaces $X$ constructed from copies of $\mathbb{H}^n_{\mathbb{F}}$ and simplicial graphs, we are able to formulate the problem of showing that groups $G, G' \leqslant \Isom(X)$ are (weakly) commensurable in topological terms.
    Let $\chi = G \backslash X$ and $\chi' = G' \backslash X$. 
    Both $\chi $ and $\chi'$ are graphs of spaces with respective fundamental groups $G$ and $G'$, and isomorphic universal covers.
    To show that $G$ and $G'$ are commensurable it therefore suffices to prove that $\chi$ and $\chi'$ have homeomorphic finite covers (see Theorem~\ref{thm:weakCommensurabilityInIdeal}).
    Note that if $G$ and $G'$ were one-ended with $X \cong \mathbb{H}^3$ and a trivial graph of groups decomposition, then constructing a common finite cover would be impossible.
    On the other hand, if $\chi$ and $\chi'$ were both graphs, with no one-ended vertex spaces isometric to $\mathbb{H}^n_\mathbb{F}$, then the existence of a common finite cover is Leighton's graph covering theorem~\cite{Leighton}.
    Our argument shows that our situation is closer to the latter than the former.
    
    We take a moment to compare this problem with work of Behrstock--Januszkiewicz--Neumann~\cite{behrstockjanuszkiewiczneumann} concerning free products of free abelian groups. 
    In contrast to the groups considered here, they prove that if $G$ and $G'$ are quasi-isometric free products of free abelian groups, then $G$ and $G'$ are abstractly commensurable. 
    Note that a finite-index subgroup of $\mathbb{Z}^n$ is isomorphic to $\mathbb{Z}^n$; equivalently, a finite-sheeted cover of a torus is still a torus.
    The fact that the volume of a closed hyperbolic manifold increases when finite covers are taken, and that the groups do not remain isomorphic, is a source of subtlety and difficulty.
    
    A key ingredient we employ in the proof of Theorem~\ref{thm:weakCommensurabilityInIdeal} is a generalization of Leighton's graph covering theorem to lattices inside {\it symmetry restricted} automorphism groups of trees. 
    Let $T$ be a locally finite simplicial tree with cocompact automorphism group $G = \Aut(T)$.
    We assume, after possibly subdividing edges or passing to an index-two subgroup of $G$ (see~\cite[Proposition 6.3]{Bass93}), that $G$ acts on $T$ without edge inversions.
    A \emph{free uniform lattice} $F \leqslant G$ is a finitely generated free subgroup that acts freely and cocompactly on $T$. 
    In the language of covering spaces, such a lattice corresponds to a finite graph $X$ and a covering map $T \rightarrow X$, where $F$ is the group of deck transformations given by $\pi_1(X)$.

    Leighton's Graph Covering Theorem~\cite{Leighton} states that any two free uniform lattices $F, F' \leqslant G$, are \emph{weakly commensurable in $G$}. That is, there exists some $g \in G$ such that $F^g \cap F'$ is a finite-index subgroup of both $F^g$ and $F'$ in $G$, where $F^g = gFg^{-1}$ (see Definition~\ref{defn:commensurabilityRelations}).
    In the language of covering spaces, this condition is equivalent to saying that any pair of finite graphs $X$ and $X'$ with isomorphic universal covers have isomorphic finite-sheeted covers.    
    Subsequent to Leighton's original proof, Bass--Kulkarni~\cite{BassKulkarni90} revisited the problem, setting it in the context of Bass-Serre theory and addressing the issue of lattice existence.
    Recently, the second author~\cite{Woodhouse18} gave a new proof, using Haar measure to solve certain gluing equations, that generalizes Leighton's theorem to \emph{graphs with fins} and has applications to a quasi-isometric rigidity result for free groups with line patterns.
    
    Walter Neumann posed a generalization of Leighton's theorem as an open problem.
    The motivation for this generalization was potential applications to quasi-isometric rigidity questions, such as generalizing Behrstock-Neumann's results for non-geometric $3$-manifolds~\cite{behrstockneumann12}.
    Shepherd and, independently (but in the appendix of the same paper), Gardam and the second author~\cite{ShepherdGardamWoodhouse}, recently solved Neumann's problem as follows, christening the generalization \emph{symmetry restricted Leighton's theorem}.
    
    Fix some $R > 0$.
    Given a vertex $v \in VT$, let $B_R(v)$ denote the closed $R$-neighborhood of $v$.
    For an element $g \in G$ and a vertex $v \in VT$, let $g_v : B_R(v) \rightarrow B_R(gv)$ denote the restriction of $g$ to $B_R(v)$.
    
    \begin{defi} \label{def:symmres}
     The \emph{$R$-symmetry restricted closure} of $H \leqslant G$ is the closed subgroup
     \[
     \mathscr{S}_R(H):= \{ g \in G \mid \forall v \in VT, \exists h \in H \textrm{ s.t. } g_v = h_v : B_R(v) \rightarrow B_R(gv) \}
     \]
     A subgroup $H \leq G$ is {\it $R$-symmetry restricted} if $H = \mathscr{S}_R(H)$.
    \end{defi}

    \begin{rema}
     Determining if a closed subgroup $H \leqslant \Aut(T)$ is a symmetry restricted group can be a subtle question.
     As discussed in~\cite[Remark A.3]{ShepherdGardamWoodhouse}, the group $\textrm{SL}_2(\mathbb{Q}_p)$ acts on its Bruhat-Tits building, which is a locally finite tree $T$, but it is not a symmetry restricted subgroup of $\textrm{Aut}(T)$ for any~$R$.
    \end{rema}
    
    \begin{thm}\cite{ShepherdGardamWoodhouse} \label{thm:symmetryRestrictedLeighton}
     Let $F, F'$ be free uniform lattices in $G$, contained in an $R$-symmetry restricted subgroup $H \leqslant G$.
     Then $F$ and $F'$ are weakly commensurable in $H$. That is to say, there exists $h \in H$ such that $F^h \cap F'$ is a finite-index subgroup of $F^h$ and $F'$. 
    \end{thm}
   
     Section~\ref{sec:commonCover} of the present paper is devoted to the challenge of arranging covers of the spaces $\chi = G \backslash X$ and $\chi' = G' \backslash X$ so that we are in a situation where Theorem~\ref{thm:symmetryRestrictedLeighton} can be applied. A key point is that Theorem~\ref{thm:symmetryRestrictedLeighton} applies only to locally finite trees and to groups that act freely and cocompactly on such a tree. While the groups $G$ and $G'$ naturally act on the Bass--Serre tree associated to the space $X$, this tree is not locally finite, and the actions are not free. In Section~\ref{sec:commonCover}, we produce a series of covering space arguments to find a common (infinite-sheeted) cover $\breve{\chi}$ of the spaces $\chi$ and $\chi'$ so that the underlying tree is locally finite and so that $\pi_1(\chi)$ and $\pi_1(\chi')$ virtually act freely on the underlying tree by deck transformations. 

    \subsection*{Acknowledgments} The authors are thankful for helpful discussions with Tullia Dymarz.
    We thank Kevin Schreve for explaining Remark~\ref{rem:proportionality}.
    The first author was supported by the Azrieli Foundation, was supported in part at the Technion by a Zuckerman Fellowship, and was partially supported by the NSF RTG grant $\#$1840190. The second author was supported by the Israel Science Foundation (grant 1026/15).

    \section{Preliminaries}

    We will use the following notation throughout the paper. 
    
    \begin{nota}
     If $A \subset X$, let $N_R(A)$ denote the open $R$-neighborhood of $A$ in $X$. Let $B_R(x)$ denote the closed ball of radius $R$ around a point $x$. If $H \leq G$ and $g \in G$, let $H^g := gHg^{-1}$. 
    \end{nota}

     The next elementary lemma can be deduced easily from standard techniques; see \cite{bowditch06,bridsonhaefliger}. 

    \begin{lemma}\label{lemma:subgroup}
        Let $X$ be a proper metric space, and let $G$ be a group which acts properly on~$X$.
        If $H \leqslant G$ acts cocompactly on $X$, then $G$ acts on $X$ cocompactly, and $H$ is a finite-index subgroup of $G$.
    \end{lemma}     
      
         There exist within the literature various notions of commensurability for subgroups and groups. Terminology can vary, so we make clear here what we mean.
 
   \begin{defi} \label{defn:commensurabilityRelations}

    \begin{enumerate}
     \item A pair of subgroups $\Gamma_1, \Gamma_2 \leqslant G$ are \emph{commensurable in $G$} if their intersection $\Gamma_1 \cap \Gamma_2$ is finite index in both $\Gamma_1$ and $\Gamma_2$.
     \item A pair of subgroups $\Gamma_1, \Gamma_2 \leqslant G$ are \emph{weakly commensurable} if there exists $g \in G$ such that the conjugate $\Gamma_1^g$ is commensurable in $G$ with $\Gamma_2$. In which case we say $g$ \emph{commensurates} $\Gamma_1$ to $\Gamma_2$.
     \item A pair of groups $\Gamma_1$ and $\Gamma_2$ are \emph{abstractly commensurable} if they have isomorphic finite-index subgroups.
     \item A pair of groups $\Gamma_1$ and $\Gamma_2$ are \emph{virtually isomorphic} if there are finite-index subgroups $H_i \leqslant \Gamma_i$ and finite normal subgroups $F_i \trianglelefteq H_i$ such that the quotients $H_1 / F_1$ and $H_2 / F_2$ are isomorphic.
    \end{enumerate}
   \end{defi}

 \subsection{The boundary of a hyperbolic space and the weak convex hull} \label{sec:gos_via_hulls}
 
  We refer the reader to~\cite{bridsonhaefliger} for background on Gromov hyperbolic spaces and their Gromov boundary.
  Let $X$ be a proper geodesic metric space, and suppose it satisfies Gromov's $\delta$-thin triangle condition.
  Associated to $X$ is its \emph{boundary} $\partial X$, a compact topological, metrizable space.
  As a set, $\partial X$ consists of geodesic rays $\gamma : [0, \infty) \rightarrow X$ up to an equivalence, where $\gamma \sim \gamma'$ if their respective images in $X$ have finite Hausdorff distance between them.
  A set in the basis for the topology on $\partial X$ is defined by fixing a ray $\gamma$ and taking all rays based at $\gamma(0)$ that fellow travel with $\gamma$ for some prescribed duration; the associated set of equivalence classes is an open set containing~$[\gamma]$.
  The group of isometries $\Isom(X)$ has an induced action on $\partial X$ by homeomorphisms.
 
  \begin{defi} \label{def:wch}
    Let $X$ be a proper geodesic hyperbolic metric space. The {\it weak convex hull} of a set $A \subset \partial X$, denoted $\textrm{WCH}_X(A)$, is the union of the geodesic lines in $X$ which have both endpoints in the subset $A$. Given a subset $S \subseteq X$, let $\Lambda S = \overline{S} \cap \partial X$ denote the {\it limit set} of $S$, where $\overline{S}$ denotes the closure of $S$ in $X \cup \partial X$. If $H \leq \Isom(X)$, then the {\it limit set} of the subgroup $H$ is $\Lambda H := \Lambda (H \cdot x)$ where $x \in X$. The limit set $\Lambda H$ does not depend on the choice of $x$.
  \end{defi}
    
  \begin{thm}\cite[Main Theorem]{swenson} \label{thm:pdc_on_hull}
   Let $G$ act properly and cocompactly by isometries on $X$.
   If $H$ is a quasi-convex subgroup of $G$, then $H$ acts properly and cocompactly on $\textrm{WCH}_X(\Lambda H) \subseteq X$.
  \end{thm}

    \subsection{The Stallings--Dunwoody decomposition}
    
    This paper concerns fundamental groups of finite graphs of groups. For background, see \cite{scottwall}, \cite{serre}. We use the following notation. 
  
    \begin{defi} \label{def:gog_gosp} 
      A \emph{graph of groups} $\mathcal{G}$ is a graph $\Gamma = (V\Gamma, E\Gamma)$ with a \emph{vertex group} $G_v$ for each $v \in V\Gamma$, an \emph{edge group} $G_e$ for each $e \in E\Gamma$, and \emph{edge maps}, which are injective homomorphisms $\Theta^{\pm}_e: G_e \rightarrow G_{\pm e}$ for each $e =(-e,+e) \in E\Gamma$.    
      A \emph{graph of spaces} associated to a graph of groups $\mathcal{G}$ is a space $Z$ constructed from a pointed \emph{vertex space} $(Z_v, z_v)$ for each $v \in V\Gamma$ with $\pi_1(Z_v,z_v) = G_v$, a pointed {\it edge space} $(Z_e, z_e)$ for each $e = (-e,+e) \in E\Gamma$ such that $\pi_1(Z_e,z_e) = G_e$, and maps $\theta^{\pm}_e: (Z_e,z_e) \rightarrow (Z_{\pm e}, z_{\pm{e}})$ such that $(\theta^{\pm}_e)_* = \Theta^{\pm}_e$. The space $Z$ is $$\left(\bigsqcup_{v \in V\Gamma} Z_v \bigsqcup_{e \in E\Gamma} \left( Z_e \times [-1,1]\right)\right) \; \Big/ \; \left\{(z,\pm 1) \sim \theta_e^{\pm}(z) \mid (z, \pm 1) \in Z_e\times [-1,1] \right\}. $$  The \emph{fundamental group} of the graph of groups $\mathcal{G}$ is $\pi_1(Z)$. The {\it underlying graph} of the graph of groups $\cG$ is the graph $\G$. A group $G$ \emph{splits as graph of groups} if $G$ is the fundamental group of a non-trivial graph of groups. 
    \end{defi}
    
    \begin{example} \label{ex:ideal_graph_of_spaces}
    Free products of closed hyperbolic manifold groups have natural graph of groups decompositions.
    If $G = \pi_1(M_1) \ast \cdots \ast \pi_1(M_k) \ast F_n$, then we make some choice of graph of groups decomposition with underlying graph $\Gamma$ and so that each vertex group $G_v$ is either the trivial group or $\pi_1(M_i)$.
    All the edge groups are trivial.
    A graph of spaces $Z$ can then be obtained by letting the vertex spaces $Z_v$ be either a point or $M_i$, and the edge space $Z_e$ also a point. 
    Indeed, in this paper we will allow all compact $M_i$ such that the universal cover $\tilde{M}_i$ is a rank-$1$ symmetric space $\mathbb{H}^n_\mathbb{F}$.
    \end{example}
    
    \begin{defi} \label{defn:ideal}
     We refer to the graph of spaces $Z$ constructed in Example~\ref{ex:ideal_graph_of_spaces} as an \emph{ideal} graph of spaces associated to a free product.
     The universal cover $\tilde{Z}$ of an ideal graph of spaces is an \emph{ideal tree of spaces}.
     If a model geometry $X$ is isometric to such a $\tilde{Z}$ then $X$ is said to be an \emph{ideal model geometry}.
     Note that the model geometry given by Proposition~\ref{prop:newModelSpace2} is an ideal model geometry.
    \end{defi}

    The graph of groups decomposition given in the next theorem is called a {\it Stallings--Dunwoody decomposition} of $G$.
     
    \begin{thm}[\cite{dunwoody85, stallings}] \label{thm:dunwoody}
      If $G$ is a finitely presented group, then $G$ splits as a finite graph of groups with finite edge groups and vertex groups that have at most one end. 
    \end{thm}
   
    The ends of a group is a quasi-isometry invariant, and any finitely presented group with more than one end has a non-trivial Stallings-Dunwoody decomposition.
    For a hyperbolic group $G$, the ends correspond to the components of $\partial G$.
    The Stallings--Dunwoody decomposition allows us to generalize from free products in Theorem~\ref{thm_first} to the quasi-isometry class of groups containing such free products (see Theorem~\ref{thm_FINAL}).
    In particular, the following theorem of Papasoglu-Whyte implies that the free product of closed hyperbolic manifold groups is quasi-isometric to a group with Stallings-Dunwoody decomposition whose one-ended vertex groups are quasi-isometric to $\mathbb{H}^n_{\mathbb{F}}$ (for possibly many different $n > 1$ and $\mathbb{F} \in \{ \mathbb{R}, \mathbb{C}, \mathbb{H}, \mathbf{Ca}\}$).
    
    \begin{thm}\cite{papasogluwhyte} \label{thm:papasogluwhyte}
     Let $G$ and $G'$ be finitely presented groups with infinitely many ends.
     The Stallings--Dunwoody decompositions of $G$ and $G'$ have the same set of quasi-isometry types of one ended vertex groups (not counting multiplicity) if and only if $G$ and $G'$ are quasi-isometric.
    \end{thm}

    Theorem~\ref{thm:papasogluwhyte} combined with the following lemma proves that any residually finite group quasi-isometric to a free product  given in the statement of Theorem~\ref{thm_first}, is virtually such a free product.
    
   \begin{lemma} \label{lemma:torsion_free}
   Let $G$ be an infinite-ended hyperbolic group such that the one-ended vertex groups in a Stallings--Dunwoody decomposition are residually finite.
   Then $G$ is residually finite and virtually torsion-free.
   \end{lemma}
   
   \begin{proof}
    Since $G$ has residually finite vertex groups and finite edge groups, the group $G$ is itself residually finite. 
    Indeed, residual finiteness is preserved under HNN extensions and amalgamated free products over finite subgroups~\cite{Baumslag63,Tretkoff73, BaumslagTretkoff78}. 
    As $G$ is a hyperbolic group, it contains only finitely many conjugacy classes of finite elements.
    Thus, after passing to a finite-index normal subgroup that excludes a finite list of representatives from each conjugacy class, one obtains a torsion-free finite-index subgroup.
   \end{proof}
    
    \begin{coro} \label{cor:qiClosed}
     Let $G = \pi_1(M_1) * \ldots * \pi_1(M_k) * \mathbb{F}_n$, where each $M_i$ is a closed hyperbolic manifold.
     Suppose that $G'$ is residually finite and quasi-isometric to $G$. 
     Then, $G'$ is virtually a free product of the form $\pi_1(M_1') * \ldots * \pi_1(M_{\ell}') * \mathbb{F}_m$, where $\widetilde{M}_i'$ is a rank-$1$ symmetric space.
    \end{coro}

    \begin{proof}
     As the number of ends is a quasi-isometry invariant, $G'$ has a non-trivial Stallings--Dunwoody decomposition.
     By Theorem~\ref{thm:papasogluwhyte} the one-ended vertex groups in this decomposition are each quasi-isometric to a rank-$1$ symmetric space $\mathbb{H}^d_\mathbb{F}$ for some $d \geq 2$.
     By rigidity results of Gabai~\cite{gabai92}, Tukia~\cite{tukia}, Chow~\cite{Chow96}, and Pansu~\cite{Pansu89}, these vertex groups act geometrically on rank-$1$ symmetric spaces.
     By Lemma~\ref{lemma:torsion_free}, the group $G'$ has a torsion-free subgroup of finite index, in which the vertex groups will embed in the isometry group of the symmetric space,  and so the induced splitting will give the desired free product decomposition. 
    \end{proof}

    \subsection{Abstract commensurability classes}
    
    We explain in this section that within the class of free products we are considering, each quasi-isometry class contains infinitely many abstract commensurability classes.
    When all the one-ended factors are cocompact Fuchsian groups this result follows from work of  Whyte~\cite[Theorem 1.6]{whyte}.
    For uniform lattices in the isometry groups of higher-dimensional rank-1 symmetric spaces, one can form incommensurable free products by forming free products of incommensurable lattices using the various means available (see, for example~\cite{MaclachlanReid, GromovPiatetski-Shapiro, NeumannReid}).
    For our purposes, there is a far simpler means of constructing incommensurable free products using a variation of Whyte's trick via the co-volume of lattices.
    
    \begin{lemma} \label{lem:infinitelyManyClasses}
     Each quasi-isometry class of free products of uniform lattices in the isometry groups of rank-1 symmetric spaces contains infinitely many abstract commensurability classes.
    \end{lemma}
    
    \begin{proof}
    The quasi-isometry class of a free product is determined by the set of quasi-isometry classes of its one-ended factors by Theorem~\ref{thm:papasogluwhyte}.
    First, suppose that we are considering a quasi-isometry class determined by a set of $n > 1$ distinct quasi-isometry classes of one-ended factors, each corresponding to a unique rank-1 symmetric space $\mathbb{H}_{\mathbb{F}_i}^{d_i}$ for $1 \leq i \leq n$. 
    Let $G = H_1 \ast \cdots \ast H_n$ and $G' = H_1' \ast \ldots \ast H_n'$ where $H_i$ and $H_i'$ are uniform lattices in $\Isom(\mathbb{H}_{\mathbb{F}_i}^{d_i})$. 
    Suppose that $G$ and $G'$ contain isomorphic finite-index subgroups $\hat{G} \leqslant G$ and $\hat{G}' \leqslant G'$. 
    Then, $\hat{G} \cong \hat{G}' = \hat{H}_1 \ast \ldots \ast \hat{H}_m \ast \mathbb{F}_k$, and for each subgroup $\hat{H}_j$ there exists an index $i$ so that $\hat{H}_j$ is conjugate in $G$ to a finite-index subgroup of $H_i$ and $\hat{H}_j$ is conjugate in $G'$ to a finite-index subgroup of $H_i'$. Let the covolumes of the lattices be $v_i = \Vol(H_i \backslash \mathbb{H}_{\mathbb{F}_i}^{d_i} )$ and $v_i' = \Vol( H_i' \backslash\mathbb{H}_{\mathbb{F}_i}^{d_i})$.
    We now assert we can express the index of the subgroups in the following terms:    
    
    \begin{claim} For each $i \in \{1, \ldots, n\}$ 
    \[
     [G : \hat{G}] = \frac{\sum_{\hat{H}_j \sim H_i} \Vol({\hat{H}_j} \backslash \mathbb{H}_{\mathbb{F}_i}^{d_i})  }{v_i}
     \;\; \textrm{ and }  \;\;
     [G' : \hat{G}'] = \frac{\sum_{\hat{H}_j \sim H_i'} \Vol({\hat{H}_j} \backslash \mathbb{H}_{\mathbb{F}_i}^{d_i})  }{v_i'},
    \]
    where $\hat{H}_j \sim H_i$ indicates that after conjugating in $G$, the subgroup $\hat{H}_j$ is a finite-index subgroup of $H_i$, and $\hat{H}_j \sim H_i'$ is similarly defined.
    \end{claim}
    
    \begin{proof}[Proof of Claim 1.]
    We prove the expression for $G$, and $G'$ follows similarly.
    Let $X$ be a graph of spaces for $G$ corresponding to the given free splitting with underlying graph $\Upsilon$, a star with root vertex $u_0$ with $X_{u_0}$ a point, each edge space a point, and each vertex space $X_{v_i}$ a presentation complex for $H_i$.
    Let $p: \hat X \rightarrow X$ be the finite cover corresponding to the subgroup $\hat G \leqslant G$.
    Then $p^{-1}(X_{v_i})$ is a disjoint collection of vertex spaces in the inducted decomposition $\{ \hat X_{\hat u_1}, \ldots, \hat X_{\hat u_k} \}$ that are in correspondence with the $\hat H_j \sim H_i$ in the free factorization of $\hat G$.
    Then the degree of $p$, which is equal to $[G : \hat G]$ can be read off by summing the degrees of each covering $\hat X_{\hat u_j} \rightarrow X_i$ which is equal to $\Vol({\hat{H}_j} \backslash \mathbb{H}_{\mathbb{F}_i}^{d_i}) / \Vol({H_i} \backslash \mathbb{H}_{\mathbb{F}_i}^{d_i})$ since covolume of a lattice is multiplicative by the degree of a finite-index subgroup.
    Thus the equality given in the statement follows by taking the sum.
    \end{proof}
    
    The summation of volumes in the numerators of the left and right hand-side are equal for each $i$. So, we deduce that for all $i$,
    \[\frac{v_i}{v_i'} = \frac{[G' : \hat{G}']}{[G : \hat{G}]}. \] 
    Therefore, choosing suitable $H_i$ and $H_i'$ to give distinct ratios for each $i$ produces infinitely many incommensurable $G$ and $G'$.
    
    In the case $n=1$, and there is a single quasi-isometry class of one-ended factors, let $G= H_1 \ast H_2$ and $G' = H_1' \ast H_2'$.
    Suppose again you can find isomorphic finite-index subgroups $\hat{G} \cong \hat{G}' \cong \hat{H}_1 \ast \ldots \ast \hat{H}_m \ast \mathbb{F}_k$. 
    Each factor $\hat{H}_j$ is conjugate in $G$ [resp. $G'$] into a factor $H_1$ or $H_2$ [resp. $H_1'$ and $H_2'$].
    We can then sum covolumes over these respective partitions of the factors to verify that
    
    \[
     [G: \hat{G}] = \frac{\sum_{\hat{H}_j^g \leq H_i} \Vol(\hat{H}_j \backslash \mathbb{H}_\mathbb{F}^d)}{v_i}
     \;\; \textrm{ and } \;\;
     [G': \hat{G}'] = \frac{\sum_{\hat{H}_j^g \leq H_i'} \Vol(\hat{H}_j \backslash \mathbb{H}_\mathbb{F}^d)}{v_i'}
    \]

    Thus we again obtain that $v_i / v_i' = [G' : \hat{G}'] / [G : \hat{G}]$ and we can obtain incommensurable $G$ and $G'$ by choosing factors with different covolume ratios.
    \end{proof}

    \section{A common simplicial model geometry} \label{sec:common_simp_geo}
  
      This section is devoted to proving Theorem~\ref{intro_nicespace}.
      We will let $G$ denote an infinite-ended hyperbolic group that is not virtually free. 
      Let $\cG$ be a Stallings--Dunwoody decomposition of $G$.
      Let $\{G_i\}_{i=1}^k$  denote the one-ended vertex groups in $\cG$.
      Let $T$ be the Bass-Serre tree for the graph of groups decomposition $\cG$.
      Suppose that $G$ acts geometrically on a proper geodesic metric space $X$, and let $H = \Isom(X)$.

     \subsection{A simplicial model geometry from intersecting weak convex hulls}
     
     As detailed by Martin--\'{S}wiatkowski~\cite[Section 2]{martinswiatkowski}, the boundary $\p G$ decomposes as 
     \[\p G \cong \bigsqcup \p G_i^g \sqcup  \p T. \]  
     The {\it atomic components} of the boundary $\p G$ are the singletons corresponding to the ends of the tree~$T$, and the {\it non-atomic components} of the boundaries are the components $\p G_i^g$, homeomorphic to the boundaries of the one-ended vertex groups. 
     The homeomorphism $\phi_G:\p G \rightarrow \p X$ induced by the geometric action of $G$ on $X$ yields well-defined {\it atomic components} and {\it non-atomic components} of $\p X$. 
     
     Let $\{ S_{\alpha} \, | \, \alpha \in I\}$ be the set of non-atomic components of $\p X$, indexed in some fashion by $I$.
     We note that $I$ is non-empty as $G$ is not virtually free.
     Let $X_\alpha = WCH(S_\al)$ be the weak convex hull of $S_{\al}$ in $X$. 
     By Theorem~\ref{thm:pdc_on_hull}, the weak convex hull $X_\alpha$  is quasi-isometric to $G_i$, a one-ended vertex group in $\cG$. 
     For $\alpha, \beta \in I$ and $r \in \R$, let
     \[ U_{\alpha}^{\beta}(r): = N_r(X_{\al}) \cap N_r(X_{\beta}). \] 
     Note that $U_{\al}^{\beta}(r) = U^{\al}_{\beta}(r)$. 
     For a specified value of $r$ we will let $U_\alpha^\beta := U_\alpha^\beta(r)$.
     
     \begin{lemma} \label{lemma_r_large}
      There exists $r \in \R$ sufficiently large so that 
      \begin{enumerate}
       \item[(i)] $X \subset \bigcup_{\alpha \in I} N_r(X_\alpha)$;
       \item[(ii)] $X_{\alpha} \subset \bigcup_{\beta \in I - \{\alpha\}} N_r(X_\beta)$; 
       \item[(iii)] $X \subseteq \bigcup_{\alpha, \beta \in I} U_\alpha^\beta(r)$.
      \end{enumerate}
      Moreover, for a given value of $r > 0$, 
      \begin{enumerate}
       \item[(iv)] there exists a constant $B = B(r) >0$ such that $\diam(U_{\alpha}^\beta(r)) < B$, and 
       \item[(v)] there is an upper bound $N = N(r)$ on the number of subsets $U_{\gamma}^{\delta}(r)$ that can intersect a given subset $U_\alpha^\beta(r)$.
      \end{enumerate}
    \end{lemma}

    \begin{proof}[Proof of Lemma~\ref{lemma_r_large}]
    If (i),  (ii), and (iii) hold for some $r > 0$ then they also hold for any $r' > r$.
    If (i) and (ii) hold for $r$, then (iii) holds for $2r$.
    Indeed,
    \begin{align*}
     X & \subseteq \bigcup_{\alpha \in I} N_r(X_{\alpha}) \\
      & \subseteq  \bigcup_{\alpha \in I} \bigcup_{\beta \in I -\{\alpha\}}  N_{2r}(X_{\alpha}) \cap N_{2r}(X_\beta) \\
      & = \bigcup_{\alpha \in I} \bigcup_{\beta \in I -\{\alpha\}} U_\alpha^\beta(2r)
    \end{align*}
     where the first line follows from (i), the second from (ii), and the third by definition.
     Since the action of $G$ preserves the sets $\{ X_\alpha \}_{\alpha \in I}$, and since the $G$ action on $X$ is cobounded, there exists $r$ such that $(i)$ holds.
     To obtain $r$ for (ii), and check (iv) and (v), we need to compare the geometry of $X$ to $G$.
     
    The group $G$ is the fundamental group of a compact graph of spaces $Z$ with associated Bass-Serre tree $T$.
     By abusing notation somewhat, we let $\tilde{Z}$ denote the $1$-skeleton of the universal cover of $Z$ equipped with the path metric, which is a model geometry for ${G}$.
     By an appropriate application of the Milnor-Schwartz lemma, there exists a ${G}$-equivariant quasi-isometry $f: \tilde{Z} \rightarrow X$. 
     After identifying $\partial \tilde{Z}$ with $\partial X$, 
     there exists $D' > 0$ such that for all $\alpha \in I$, 
     the weak convex hull $\tilde{Z}_\alpha$ of $S_\alpha$ is within Hausdorff distance $D'$ from a vertex space $\widetilde{Z}_v$ in $\widetilde{Z}$. 
     There are corresponding subsets $V_\alpha^\beta$ in $\tilde{Z}$ defined in a similar fashion to $U_\alpha^\beta$ for some sufficiently large constant~$R$.
        We will show the claims hold for the subsets $V_\alpha^\beta$ in $\tilde{Z}$ and then deduce they hold in $X$ as well. 
    
       To obtain (ii) for $\{V_\alpha^\beta\}$, observe that each vertex space $Z_v$ is contained in a finite neighborhood of its incident edge spaces, so we can choose $R$ to be sufficiently large such that 
       $$\tilde Z_v \subseteq \bigcup_{(u,v) \in ET} N_R(\tilde Z_u).$$
       Therefore, using the existence of the constant $D'$, we can ensure that $R$ is large enough that
       $$
       \tilde Z_\alpha \subseteq \bigcup_{\beta \in I -\{\alpha\}} N_R(\tilde Z_\beta).
       $$
       
       For (iv), to obtain the upper bound on the diameter of $V_{\alpha}^\beta$, we let $\tilde Z_v$ be the vertex space Hausdorff distance $D'$ from $\tilde Z_\alpha$, and $\tilde Z_u$ be the vertex space Hausdorff distance $D'$ from $\tilde Z_\beta$.
       Let $e_1$ be the edge in $T$ incident to $v$ that is closest to $u$, and let $e_2$ be the edge in $T$ incident to $u$ that is closest to $v$.
       Then
       \begin{align*}
        V_\alpha^\beta & = N_R(\tilde Z_\alpha) \cap N_R(\tilde Z_\beta) \\
                       & \subseteq N_{R + D'}(\tilde Z_v) \cap N_{R+D'}(\tilde Z_u) \\
                       & \subseteq N_{R + D'}(\tilde Z_{e_1}) \cup N_{R+D'}(\tilde Z_{e_2}).
       \end{align*}
       To see the final inequality, observe that if $x$ lies in $N_{R + D'}(\tilde Z_v) \cap N_{R+D'}(\tilde Z_u)$ then there are paths $\gamma_1$ and $\gamma_2$ of length at most $R+D'$ that respectively connect $\tilde Z_v$ to $x$ and $\tilde Z_u$ to $x$.
       By considering the tree of spaces decomposition we conclude that either $\gamma_1$ must pass through $\tilde Z_{e_1}$ or $\gamma_2$ must pass through $\tilde Z_{e_2}$, so $x$ lies in the $R+D'$ neighborhood of one of these edge spaces.
       Since edge spaces in $\tilde Z$ have finite diameter, (iv) follows.
       
    For (v) suppose that for all $m >0$ there is some $V_\alpha^\beta$ such that $V_{\gamma_i}^{\delta_i} \cap V_\alpha^\beta \neq \emptyset$ for distinct sets $ V_{\gamma_i}^{\delta_i}$ with $1 \leq i \leq m$ and $\gamma_i$ distinct from $\alpha$ or $\beta$.
         Then since all these subsets have diameter bounded above by a constant $B'$, we deduce that the subsets $\{V_{\gamma_i}^{\delta_i}\}$ are pairwise within distance $B'$ of each other.
         This implies that the subspaces $\{\tilde Z_{\gamma_i}\}$ are pairwise distance at most $B' + 2R$ apart.
         If $\tilde Z_{v_i}$ is the vertex space Hausdorff distance $D'$ from $\tilde Z_{\gamma_i}$, then we conclude that $\{\tilde Z_{v_i}\}$ are pairwise distance at most $B' + 2R+ 2D'$ apart.
         As $m$ was arbitrary, this violates the bounded packing of the one-ended vertex groups $\{G_{\gamma_i}\}$; see~\cite{HruskaWise09} for definition and details of bounded packing.
    
    We now deduce (ii), (iv), and (v) for $X$, using standard quasi-isometry and $\delta$-hyperbolicity arguments, which we include for the benefit of the reader.
    The sufficient size of the constant $r \geq 0$ is determined by $R$, the quasi-isometry constants $(\lambda, \epsilon)$ for the map $f: \tilde Z \rightarrow X$, and $\delta \geq 0$ large enough so that both $X$ and $\widetilde{Z}$ are $\delta$-hyperbolic.
    First, there exists a constant $D = D(\lambda,\epsilon,\delta)$ so that for all $\alpha \in I$, the Hausdorff distance between $f(\tilde{Z}_\alpha)$ and $X_\alpha$ is at most $D$.
    Indeed, by the Morse Lemma~\cite[Theorem III.H.1.7]{bridsonhaefliger}, there exists a $D_1 = D_1(\lambda, \epsilon, \delta)$ sufficiently large such that a bi-infinite $(\lambda,\epsilon)$-quasi-geodesic lies at distance at most $D_1$ from an actual geodesic. 
    Thus, $f(\tilde{Z}_\alpha) \subseteq N_{D_1}(X_\alpha)$. 
    For the converse inclusion we consider a quasi-inverse $\bar f$ of $f$, which has quasi-isometry constants $( \bar \lambda, \bar \epsilon)$ determined by $(\lambda,\epsilon)$.
    Thus, by the Morse Lemma, there exists $D_2 = D_2(\lambda, \epsilon, \delta) > 0$, such that $\bar{f}(X_\alpha) \subseteq N_{D_2}(\tilde{Z}_\alpha)$.
    Then, by the definition of quasi-inverse, there exists a constant $k = k(\lambda, \epsilon)$ so that $$ X_\alpha \subseteq N_k\bigl(f \circ \bar{f}(X_\alpha)\bigr) \subseteq N_k\bigl(f(N_{D_2}(\tilde{Z}_\alpha))\bigr).$$
    so $X_\alpha \subseteq N_D(f(\tilde{Z}_\alpha))$ if $D > \lambda D_2 + \epsilon +k$. Thus we take $D > \max\{ D_1, \lambda D_2 + \epsilon +k \}$.
    
    Since (ii) holds for $\tilde{Z}$,
    we can deduce that
    \begin{align*}
     X_\alpha & \subseteq N_D(f(\tilde Z_\alpha)) \\
              & \subseteq N_D\Bigg(\bigcup_{\beta \in I - \{ \alpha \} } f(N_R(\tilde Z_\beta)) \Bigg)\\
              & \subseteq N_D\Bigg(\bigcup_{\beta \in I - \{ \alpha \} } N_{\lambda R + D +\epsilon}(X_\beta) \Bigg),
    \end{align*}
    so (ii) holds for $X$ with $r > \lambda R + 2D +\epsilon$.
      
      To see (iv) and (v) for $X$ observe that
      \begin{align*}
       \bar f(U_\alpha^\beta(r)) & = \bar f(N_r(X_\alpha) \cap N_r(X_\beta)) &\\
                            & \subseteq \bar f(N_r(X_\alpha)) \cap \bar f(N_r(X_\beta)) & \\
                            & \subseteq N_{\bar \lambda r + \bar \epsilon}(\bar f(X_\alpha)) \cap N_{\bar \lambda r + \bar \epsilon}(\bar f(X_\beta)) \\
                            & \subseteq N_R(\tilde Z_\alpha) \cap N_R( \tilde Z_\beta) = V_\alpha^\beta(R)
      \end{align*}
      where $R  > \bar \lambda r + \bar \epsilon +D_2$.
      If $\diam(U_\alpha^\beta(r))$ were unbounded over $\alpha,\beta \in I$, then $\diam(V_\alpha^\beta(R))$ would be unbounded too.
      If $U_\alpha^\beta(r)$ intersects arbitrarily large collections of distinct sets $U_\gamma^\delta(r)$, then $V_\alpha^\beta(R)$ would intersect the corresponding arbitrarily large collection of distinct sets of the form $V_\gamma^\delta(R)$.
    \end{proof}

     Define a (1-dimensional) simplicial complex $\cY$ as follows. Let $r \in \R$ be sufficiently large, as to satisfy the conclusions of Lemma~\ref{lemma_r_large}. 
     Let $\cY$ be the $1$-skeleton of the nerve of the cover of $X$ by the sets $\{ U_\al^\beta \, | \, U_\al^\be \neq \emptyset\}$.
     That is, the vertices of $\cY_{\alpha}$ are in one-to-one correspondence with elements of the set $\{ U_\al^\be \, | \, U_\al^\be \neq \emptyset\}$, and there is an edge $\{ U_{\al}^\be, U_\gamma^\delta \}$ if and only if $U_\alpha^\beta \cap U_\gamma^\delta \neq \emptyset$.
     We view $\cY$ as a metric space by equipping each edge with the Euclidean metric, such that each edge has length one. 
     As $H$ acts by isometries on $X$ and permutes the non-atomic components of $\p X$, we deduce that $H$ preserves the open cover $\{ U_\al^\beta \, | \, U_\al^\be \neq \emptyset\}$ so there is an induced $H$-action on $\cY$. 
 
      \begin{lemma} \label{lemma_y_qi}      
        The graph $\cY$ is connected and locally finite. The group $G$ acts geometrically on $\cY$.
        The $H$-action on $X$ is \emph{quasi-conjugate} to the $H$-action on $\cY$.
        That is, there is a quasi-isometry $\phi : X \rightarrow \cY$ and a constant $L >0$ such that for all $x\in X$ $$ d_{\cY}(h \cdot \phi(x), \phi(h\cdot x)) < L.$$ 
       \end{lemma}
       
       The existence of the quasi-conjugacy will be due to the following, more general lemma:
       
       \begin{lemma} \label{lem:MSenvelope}
 Let $H$ be a group acting on proper geodesic metric spaces $X$ and $Y$.
 Let $G \leqslant H$ be a subgroup such that the restrictions of the action of $H$ to $G$ on $X$ and $Y$ are geometric.
 Then, there is an $H$-equivariant quasi-isometry $f: X \to Y$.
\end{lemma}

\begin{proof}
 The \v{S}varc-Milnor Lemma gives $G$-equivariant quasi-isometries:
 \[
  f_X : \Cay(G,S) \to X \; \; \; \textrm{ and } \; \; \; f_Y : \Cay(G, S) \to Y. 
 \]
 So, there is a $G$-equivariant quasi-isometry $f_Y \circ \bar{f}_X: X \rightarrow Y$, where $\bar{f}_X$ is a $G$-equivariant quasi-inverse for $f_X$. 
 However, this map is not necessarily $H$ equivariant.
 
 Let $x_0 = f_X(e)$.
 There exists $r >0$ such that the $r$-neighbourhood of $G\cdot x_0$ covers $X$, and there exists $R>0$ so that $f_Y \circ \bar{f}_X(N_r(x_0)) \subseteq N_R(f_Y(e))$.
 For each $H$-orbit in $X$, choose a representative $x\in N_r(x_0)$.
 Define $f(hx) := h \bigl( f_Y \circ \bar{f}_X(x) \bigr)$.
 Then $f$ is $H$-equivariant, and it is coarsely equivalent to $f_Y \circ \bar{f}_X$ since if
 $x \in X$, then there is $g \in G$ such that $gx \in N_r(x_0)$, so
 \[
  d(f(x), f_Y \circ \bar{f}_X(x)) = d(f(gx), f_Y \circ \bar{f}_X(gx)) \leq 2R.
 \]
\end{proof}

       \begin{proof}[Proof of Lemma~\ref{lemma_y_qi}]
        The space $\cY$ is connected by \emph{(iii)} of Lemma~\ref{lemma_r_large}, and locally finite by \emph{(v)} of Lemma~\ref{lemma_r_large}. 
        To see that $G$ acts properly on $\cY$ observe that $\Stab_G(U_\alpha^\beta(r))$ stabilizes a finite diameter set in $X$, by \emph{(iv)} of Lemma~\ref{lemma_r_large}, so the properness of $G$ on $X$ implies that $\Stab_G(U_\alpha^\beta(r))$ is finite, and thus the action on $\cY$ is proper.
        Cocompactness follows by observing that cocompactness of $G$ on $X$ implies that there exists some compact set $K \subseteq X$ such that $G K = X$.
        Thus every $G$-orbit of $U_\alpha^\beta$ has at least one representative intersecting $K$.
        By \emph{(v)} of Lemma~\ref{lemma_r_large}, only finitely many $U_\alpha^\beta(r)$ intersect $K$, which allows us to conclude that there are only finitely many vertex orbits in $\cY$, so cocompactness of the $G$ action follows from local finiteness of $\cY$.
        Finally the actions are quasi-conjugate by Lemma~\ref{lem:MSenvelope}.
       \end{proof}

      \begin{lemma}       
       \begin{enumerate} 
        \item The graph $\cY$ is connected and locally finite.
        \item The $H$-action on $X$ is \emph{quasi-conjugate} to the $H$-action on $\cY$.
        That is, there is a quasi-isometry $\phi : X \rightarrow \cY$ and a constant $B >0$ such that for all $x\in X$ $$ d_{\cY}(h \cdot \phi(x), \phi(h\cdot x)) < B.$$ 
       \end{enumerate}
       Consequently, $G$ acts geometrically on $\cY$. 
      \end{lemma}
      
      \begin{proof}
       The space $\cY$ is connected by \emph{(iii)} of Lemma~\ref{lemma_r_large}, and locally finite by \emph{(v)} of Lemma~\ref{lemma_r_large}. 
        
        We construct a quasi-isometry $\phi$ as follows: for $x \in X$ let $\phi(x)$ be the vertex corresponding to some choice of $U_{\alpha}^\beta$ containing $x$.
        (This is possible by \emph{(iii)} in Lemma~\ref{lemma_r_large}.)
        To verify that $\phi$ is a quasi-isometry we first observe that $$d_X(x, x') \leq 2rd_\cY(\phi(x), \phi(x')) +2r$$ for $x,x' \in X$.
        Indeed, let $n = d_\cY(\phi(x), \phi(x'))$ and $\phi(x) = v_0, \ldots, v_n = \phi(x')$ are the vertices in a geodesic, where $v_i$ corresponds to $U_{\alpha_i}^{\beta_i}$. Then, since $U_{\alpha_i}^{\beta_i} \cap U_{\alpha_{i+1}}^{\beta_{i+1}} \neq \emptyset$ and by the upper bound $\diam(U_{\alpha}^\beta) < B$, we can construct a path of length at most $2rn+2r$ from $x \in U_{\alpha_1}^{\beta_1}$ to $x' \in U_{\alpha_n}^{\beta_n}$.
        
        To obtain an upper bound on $d_\cY(\phi(x), \phi(x'))$ in terms of $d_X(x, x')$, observe that the action of $G$ on $X$ is cocompact, and $\{U_{\alpha}^{\beta}\}$ is an open cover of $X$, so we can apply Lesbesgue's number lemma to obtain $\ell> 0$  such that each subset of diameter less than $\ell$ is contained in some $U_\alpha^\beta$.
        We can then subdivide a geodesic joining $x$ to $x'$ into $m := d_X(x,x')\slash \ell$ intervals.
        Let $U_{\alpha_i}^{\beta_i}$ contain the $i$-th interval and observe that the vertex given by $\phi(x)$ must be adjacent to $U_{\alpha_1}^{\beta_1}$ and the vertex given by $\phi(x')$ must be adjacent to $U_{\alpha_m}^{\beta_m}$.
        Thus, we conclude that
        \[
         d_\cY(\phi(x), \phi(x')) \leq \frac{1}{\ell}d_X(x,x') + 2
        \]
    Finally, we deduce that $\phi$ is quasi-surjective, since given a vertex corresponding to the subset $U_\alpha^\beta$, if $x \in  U_\alpha^\beta$, then $\phi(x)$ is adjacent to $U_\alpha^\beta$. 
    Thus $\tilde{Z}$ is contained in the $1$-neighborhood of $\phi(X)$.

    To see that $\phi$ quasi-conjugates the respective $H$-actions, let $x \in X$ and let $\phi(x) = U_\alpha^\beta$ and let $\phi(h\cdot x) = U_\gamma^\delta$.
    Then since $h \cdot x \in hU_\alpha^\beta \cap U_\gamma^\delta$ we conclude that $h \cdot \phi(x)$ and $\phi(h\cdot x)$ are adjacent, so we can let $B =2$. Since $G$ acts geometrically on $X$, the group $G$ acts geometrically on $\cY$. 
      \end{proof}

      We extend $\cY$ to a simply connected $2$-complex  $\hat \cY$ as follows. 
      The Milnor-Schwartz lemma combined with Lemma~\ref{lemma_y_qi} proves that the space $\cY$ is quasi-isometric to $G$ and is therefore a Gromov hyperbolic space.
      Let $P_D(\cY)$ denote the \emph{Rips complex} of $\cY$: the simplicial complex with vertex set $\cY^{(0)}$, the vertex set of the graph $\cY$, and $n$-simplices given by all subsets of vertices of $\cY$ of diameter at most $D$. 
      As $\cY$ is $\delta$-hyperbolic, for all $D > 4\delta +3 $ the complex $P_D(\cY)$ is a locally finite, contractible complex~\cite[III.$\Gamma$ Prop 3.23]{bridsonhaefliger}.  
      Let $\hat \cY$ be the $2$-skeleton of $P_D(\cY)$ for some $D > 4 \delta + 3$.
      The $H$-action on $\cY$ extends to $P_D(\cY)$ and $\hat \cY$, and the embedding of $\cY$ into $\hat \cY$ is a quasi-isometry that will quasi-conjugate the respective actions.
      We let $\hat \phi : X \rightarrow \hat \cY$ denote the resulting quasi-conjugacy of the $H$-actions.

   \subsection{Tracks and a new tree of spaces model geometry}
   
    We employ the notion of \emph{tracks}, as first defined by Dunwoody~\cite{dunwoody85}. 
  
	\begin{defi}
	Let $K$ be a simplicial $2$-complex.
	 A {\it track} $\tau$ is a connected subset of $K$ (or rather its topological realization) such that
	 \begin{enumerate}
	  \item for each $2$-simplex $\sigma$ of $K$, the intersection of $\tau$ with $\sigma$ is the union of finitely many disjoint straight lines joining distinct edges of $\sigma$.
	  \item If $e$ is a $1$-simplex in $K$ not contained in a $2$-simplex, then either $\tau$ does not intersect $e$, or $\tau$ is a single point in the interior of $e$. 
	 \end{enumerate}
	\end{defi}

	We will employ the following theorem of Dunwoody~\cite{dunwoody85} in the form stated by Mosher--Sageev--Whyte~\cite{moshersageevwhyteI}. Alternatively, consult~\cite[Chapter 20]{DrutuKapovich}.

    \begin{thm} \cite{dunwoody85} \cite[Theorem 15]{moshersageevwhyteI} \label{thm:dunwoodyTracks}
     Let $K$ be a locally finite, simply connected, simplicial $2$-complex with cobounded isometry group.
     There exists a disjoint union of finite tracks $\tau = \bigsqcup \tau_i$ in $K$ invariant under the action of $\Isom(K)$ such that the closure of each component of $K - \tau$ has at most one end.
    \end{thm}
    
    \begin{proof}[Proof of Theorem~\ref{intro_nicespace}]
     As $\hat{\cY}$ is a locally finite, simply connected, simplicial $2$-complex with cobounded isometry group, Theorem~\ref{thm:dunwoodyTracks} yields an $H$-equivariant set of finite tracks $\tau = \bigsqcup_{i \in J} \tau_i$ such that the closure of each component $\hat \cY - \tau$ has at most one end.
     Let $\cT$ be the dual tree to this set of tracks. 
     Each vertex $v$ in $\cT$ corresponds to a component $\hat \cY_v$ of $\hat \cY - \tau$.
     Reindex the set of tracks so that each edge $e$ in $\cT$ corresponds to a track $\tau_e$ in $\tau$.
     
     Construct an $H$-equivariant map $q: \hat \cY \rightarrow \cT$ as follows.
     For each track $\tau_e$ there exists a product neighborhood of $\tau_e$ in $\hat \cY$ homeomorphic to $\tau_e \times [0,1]$.
     These product neighborhoods can be chosen $H$-equivariantly and disjoint from the set of $0$-simplices and from each other to obtain an $H$-equivariant set $\tau \times [0,1] = \bigsqcup_{e \in E\cT} \tau_e \times [0,1]$.
     Each component of $\hat \cY - \left(\tau \times [0,1]\right)$ is a subspace $\hat \cY_v' \subseteq \hat \cY_v$.
     Note that $\hat \cY_v'$ has the same number of ends as $\hat \cY_v$, as each $0$-simplex in $\hat \cY_v$ is a $0$-simplex in $\hat \cY_v'$ as well.
     Define the $H$-equivariant map $q$ so that $q(\hat \cY_v') = v$ and $\tau_e \times [0,1]$ is mapped to $e$ by projection onto the second factor.
     Then, the map $q$ decomposes the space $\hat \cY$ as a tree of spaces.  
      
       \begin{figure}
	\begin{overpic}[width=.7\textwidth,tics=5, ]{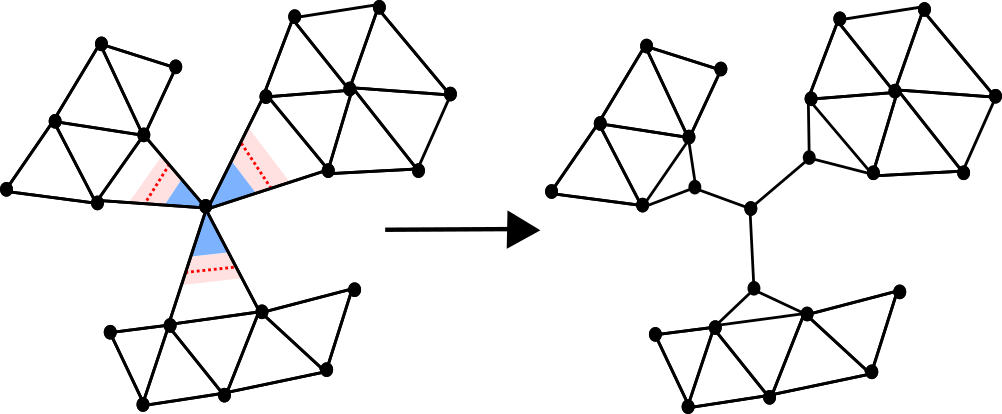} 
 	   \put(7.3,17.6){$\tau \times [0,1]$}
 	   \put(18.6, 24.2){$\hat \cY_v'$}
 	   \put(73.5,23.1){$Y_v$}
 	   \put(42,14){$\hat{q}$}
          \end{overpic}
	\caption{An illustration of the map $q$. The dashed lines on the left denote tracks and their shaded neighborhoods are collapsed to edges on the right. The blue region around the central vertex is the space $\hat {\cY}_v'$ and is collapsed to the vertex $Y_v$ on the right.}
	\label{fig:collapse}
    \end{figure}
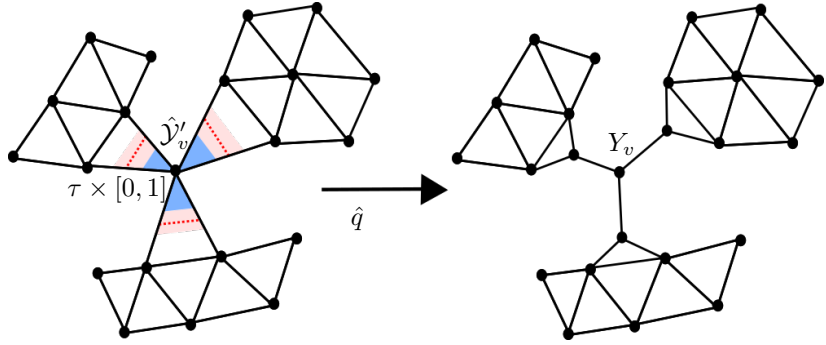
     
     Let $Y$ be obtained from $\cY$ by collapsing the vertex spaces $\hat \cY_v'$ with zero ends to vertices and collapsing each subspace $\tau_e \times [0,1]$ to an edge.
     Then, the map $q$ can be factored $\hat \cY \rightarrow Y \rightarrow \cT$, where the first map $\hat{q}: \hat \cY \rightarrow Y $ is given by collapsing as above.
     See Figure~\ref{fig:collapse}.
     The $1$-ended vertex spaces $\hat \cY_v'$ have their (finite) intersections $\hat \cY_v' \cap \tau_e \times [0,1]$ crushed to points to obtain new $1$-ended vertex spaces $ Y_v$.
     Indeed, the $0$-simplices in $\hat \cY_v'$ embed in $Y_v$.
     
     The quotient map $\hat{q}: \hat \cY \rightarrow Y$ is $H$-equivariant and obtained by collapsing sets of universally bounded diameter, so $\hat{q}$ is a quasi-conjugacy.
     By composing the $H$-quasi-conjugacies $\hat q \circ \hat \phi : X \rightarrow \hat \cY \rightarrow Y$ we obtain the quasi-conjugation from the statement of the theorem.
    \end{proof}

    \section{A common hyperbolic model geometry} \label{sec:hyp_model_geo}
    
    To obtain an ideal common model geometry that is built out of copies of rank-1 symmetric spaces, we will apply the next theorem. 
   
    \begin{thm} \label{thm:actionPromotion}
     Let $H$ act cocompactly on a proper geodesic metric space $X$ which is quasi-isometric to a rank-$1$ symmetric space $\mathbb{H}^n_\mathbb{F}$.
     Then, $H$ acts cocompactly on $\mathbb{H}^n_{\mathbb{F}}$ and there is a quasi-isometry $f: X \rightarrow \mathbb{H}^n_{\mathbb{F}}$ and a constant $D \geq 0$ such that $d(h \cdot f(x), f(h\cdot x)) < D$ for all $x \in X$ and $h \in H$.  
    \end{thm}

   The natural language for proving Theorem~\ref{thm:actionPromotion} is that of \emph{quasi-actions}.
   For background on quasi-isometries and quasi-actions we refer the reader to  Drutu--Kapovich~\cite[Section 8.5]{DrutuKapovich}.

  \begin{defi} \label{defn_quasiaction}
    Let $G$ be a group and $X$ a metric space. An {\it $(L,A)$-quasi-action} of $G$ on $X$ is a map $\phi: G \rightarrow Map(X,X)$ such that 
    \begin{enumerate}
     \item $\phi(g)$ is an $(L,A)$-quasi-isometry of $X$ for all $g \in G$;
     \item $d(\phi(id), Id_X) \leq A$;
     \item $d(\phi(g_1g_2), \phi(g_1)\phi(g_2)) \leq A$ for all $g_1, g_2 \in G$. 
    \end{enumerate}
    A quasi-action is {\it cobounded} if there exists $x \in X$ and a constant $R$ such that for all $x' \in X$ there exists $g \in G$ so that $d(x', \phi(g)(x)) \leq R$.
  \end{defi}

  Given an action of a group $G$ on a geodesic metric space $X$ and a quasi-isometry $q: X \rightarrow Y$ to another geodesic metric space, one obtains a \emph{conjugate quasi-action} $\phi: G \rightarrow Map(Y,Y)$ given by $\phi(g) = q \circ g \circ \bar{q}$ where $\bar{q}$ denotes a quasi-inverse of $q$. 
  The conjugate quasi-action yields the following lemma, a version of the \emph{quasi-action principal}.
  
  \begin{lemma} \label{lem:quasiActionPrincipal}
     If $G$ is a group that acts by isometries on a metric space $X$, and $X$ is quasi-isometric to a metric space $Y$, then there is a quasi-action of $G$ on $Y$. Moreover, if the action of $G$ on $X$ is cocompact, then the quasi-action of $G$ on $Y$ is cobounded. 
  \end{lemma}
    
    Lemma~\ref{lem:quasiActionPrincipal} applied to a group $H$ acting cocompactly by isometries on a geodesic metric space quasi-isometric to $\mathbb{H}^n_{\mathbb{F}}$ yields a cobounded quasi-action of $H$ on $\mathbb{H}^n_{\mathbb{F}}$.
    
   Theorem~\ref{thm:actionPromotion} is thus the consequence of the following statement, a special case of Theorem 1.4 in~\cite{KleinerLeeb09} (see also~\cite{KleinerLeeb01}), that is the consequence of a wide body of work~\cite{Pansu89, tukia, Chow96, hinkkanen85, hinkkanen90, markovic06, gabai92, CassonJungreis94}.
    We refer the reader to~\cite{DrutuKapovich} for recent exposition on the real hyperbolic case.     
    
    \begin{thm} 
     A cobounded quasi-action $\phi$ of a group $G$ on the rank-1 symmetric space $\mathbb{H}^n_{\mathbb{F}}$ is quasi-isometrically conjugate to an isometric action.
    \end{thm}
      
    We can  now apply Theorem~\ref{thm:actionPromotion} to build a new common ideal model geometry as in Definition~\ref{defn:ideal}.
    
     \begin{prop} \label{prop:newModelSpace2}
      Let $G$ be an infinite-ended group with a Stallings-Dunwoody decomposition in which every one-ended vertex group $G_v \leq G$ is quasi-isometric to a rank-$1$ symmetric space $\mathbb{H}^{n(v)}_{\mathbb{F}}$ for some $n(v) \geq 2$, and such that there is at least one such vertex group. 
      If $X$ is a model geometry for $G$, and $H = \Isom(X)$, then the $H$-action on $X$ is quasi-conjugate to an $H$-action on an ideal model geometry. 
     \end{prop}
     
     \begin{proof}    
      Apply Theorem~\ref{intro_nicespace} to quasi-conjugate the $H$-action on $X$ to a simplicial $H$-action on a simplicial $2$-complex $Y$ that decomposes as a tree of spaces with underlying graph $T$ and with vertex spaces isomorphic to either points or one-ended simplicial $2$-complexes, and each edge space a point.
      The group $H$ acts on the tree $T$.
      
      Let $Y_v$ be a one-ended vertex space of $Y$. 
      The space $Y_v$ is quasi-isometric to $\mathbb{H}^n_{\mathbb{F}}$ for some $n \geq 2$ since it is stabilized by a one-ended subgroup in the Stallings-Dunwoody decomposition of $G$.
      Let $H_v = \Stab_H(Y_v)$. 
      The group $H_v$ acts on $Y_v$ cocompactly since it acts on $Y_v$ simplicially and contains a subgroup $G_v$ that acts on $Y_v$ cocompactly.
      Apply Theorem~\ref{thm:actionPromotion} to obtain a cocompact $H_v$-action on $\mathbb{H}^n_{\mathbb{F}}$ and a quasi-isometry $f_v : Y_v \rightarrow \mathbb{H}^n_{\mathbb{F}}$ that quasi-conjugates the action of $H_v$ on $Y_v$ to the action of $H_v$ on $\mathbb{H}^n_{\mathbb{F}}$.
      
      To obtain the ideal model geometry we will equivariantly remove the $H$-orbit of $Y_v$ and replace it with a copy of $\mathbb{H}^n_{\mathbb{F}}$. 
      The $H$-orbit of $Y_v$ is the disjoint union of vertex spaces, each of the form $Y_{hv}$ for some $h \in H$. 
      Enumerate the vertices $H\cdot v = \{ v= v_0 , v_1, v_2, \ldots \}$. 
      For each $i \in \N$, choose $h_i \in H$ such that $h_0 = id$ and $h_i \cdot v = v_i$; then $\{h_i\}_{\N}$ is a set of coset representatives for $H/H_v$.
      Realize $H \cdot Y_v$ as the direct product $Y_v \times (H / H_v)$ with induced $H$-action given by
      $$ h \cdot (y, [h_i]) = ( h_j^{-1} h h_i \cdot y, [h_j] = [hh_i]).$$
      Note that $[h_j] = [hh_i]$ implies that $h_j^{-1} h h_i \in H_v$.
      This action on the product is the same as the natural action of $H$ on $H\cdot Y_v$, after identifying the element $h_i Y_v$ with the element $(Y_v, [h_i])$ via the isomorphism  that maps $h_i \cdot y \mapsto (y, [h_i])$, where $y \in Y_v$.
      
      Now, define an action of $H$ on $\mathbb{H}^n_{\mathbb{F}} \times (H / H_v)$ by
      $$ h \cdot (x, [h_i]) = ( h_j^{-1} h h_i \cdot x, [h_j] = [hh_i]).$$
      Take the closure of $Y - H \cdot Y_v$ and the disjoint union
      \[
       \overline{Y - H \cdot Y_v} \sqcup (\mathbb{H}^n_{\mathbb{F}} \times H / H_v)
      \]
      to recover a tree of spaces by equivariantly reattaching the ends of the edge spaces that intersected $H \cdot Y_v$ and to $(\mathbb{H}^n_{\mathbb{F}} \times H / H_v)$ as follows.      
      Let $e$ be an edge of the tree $T$ incident to the vertex $v$. Let $Y_e$ denote the corresponding edge space in $Y$ and $H_e = \Stab_H(Y_e)$. 
      Each edge of $H_v \cdot e$ corresponds to a coset in $H_v / H_e$.
      Let $y_e = Y_v \cap Y_e$.
      The $H_e$-orbit of $f_v(y_e) \in \Hy_{\mathbb{F}}^n$ is a bounded set, since $f_v$ quasi-conjugates the action of $H_v$ on $Y_v$ to the action of $H_v$ on $\Hy_{\mathbb{F}}^n$.
      Thus, the convex hull of $H_e\cdot f_v(y_e)$ has a center $\hat{y}_e$ invariant under $H_e$; see \cite[Proposition II.2.7]{bridsonhaefliger}.
      The $H$-orbit of $(\hat{y}_e, [id])$ in $\mathbb{H}^n_{\mathbb{F}} \times (H / H_v)$ defines the points to which the the edge spaces $H \cdot Y_e$ are reattached.
      More precisely, attach the endpoint $h\cdot y_e$ of the edge $h\cdot e \in \overline{Y - H \cdot Y_v} $ to the point $h \cdot (\hat{y}_e, [id]) \in \mathbb{H}^n_{\mathbb{F}} \times H/H_v$.
      Doing this for all $H_v$-orbits of edges incident to $v$, we obtain a new tree of spaces as a quotient space 
      \[
       \left(\overline{(Y - H \cdot Y_v)} \cup (\mathbb{H}^n_{\mathbb{F}} \times H / H_v)\right) / \sim
      \]
      which has a natural $H$-action and such that $G$ and $G'$ act geometrically.
      Repeating for all $H$-orbits of one-ended vertex spaces yields the desired ideal model geometry.
      \end{proof}
    
    \subsection{Commensurability of certain manifold groups}  
      
    \begin{lemma} \label{lemma:discrete_geometric}
    If $\cP \subset \fhyp$ is a discrete subset stabilized by a uniform lattice $\Gamma$, then $\Gamma$ is a finite-index subgroup of $\Stab(\cP) \leqslant \Isom(\fhyp)$.
    \end{lemma}
    \begin{proof}
    Equip the subspace $\cP$ with the metric induced by the inclusion $\cP \hookrightarrow \fhyp$. Since $\cP$ is a proper metric space, it suffices to prove that $\Gamma' = \Stab(\cP) \leq \Isom(\fhyp)$ acts properly on $\cP$ by Lemma~\ref{lemma:subgroup}. Let $y \in \mathcal{P}$ and pick $\epsilon>0$ so that $B_\epsilon(y)$ contains no other points of $\mathcal{P}$. 
    Then, if $\gamma \in \Gamma'$ so that $d(y, \gamma\cdot y) < \epsilon$, then $\gamma \cdot y = y$. Thus, it is enough to prove that $\Stab_{\Gamma'}(y)$ is finite. 
    Choose $R>0$ so that $B_R(y)$ contains at least $m+1$ points in $\mathcal{P}$ that are not contained in a {codimension-1} hyperplane, where $m$ is the (real) dimension of $\mathbb{H}^m_{\mathbb{F}}$.
     There is a homomorphism $\Stab_{\Gamma'}(y) \rightarrow \textrm{Sym}({M})$, the group of permutations of $M > m$ elements. 
    Any element of the kernel fixes $m+1$ distinct points in $\mathbb{H}_{\mathbb{F}}^m$ not contained in a codimension-$1$ hyperplane, so the kernel consists of only the trivial isometry.
    Indeed, the stabilizer of a single point will act by isometries on the unit tangent space which is an $(m-1)$-sphere.
    The remaining fixed points will correspond to fixed points in the unit tangent space.
    Since the remaining fixed points in $\mathbb{H}_{\mathbb{F}}^m$ don't stabilize a codimension-1 subspace, the corresponding fixed points in the unit tangent sphere will give a basis for the tangent space, and therefore the entire tangent sphere will be fixed and the isometry will be trivial. 
    Thus, $\Stab_{\Gamma'}(y)$ is finite, as desired. 
    \end{proof}
    
    The next proposition follows from the previous lemma and arguments tacit in the proof of Proposition~\ref{prop:newModelSpace2}.  

    \begin{prop} \label{prop:simp_ac}
    If $G$ and $G'$ are closed hyperbolic manifold groups that act geometrically on the same simplicial complex, then $G$ and $G'$ are virtually isomorphic.
    Moreover, if $G$ and $G'$ are residually finite then $G$ and $G'$ are abstractly commensurable. 
    \end{prop}
    \begin{proof}
        Suppose $G$ and $G'$ are closed hyperbolic manifold groups that act geometrically on the same simplicial complex $X$, which is quasi-isometric to $\Hy_{\mathbb{F}}^n$. Let $H = \Isom(X)$. The group $H$ acts on $X$ cocompactly since $X$ is a simplicial complex and the subgroups $G,G' \leq H$ act on $X$ cocompactly. There exists a quasi-isometry $f:X \rightarrow \fhyp$ that quasi-conjugates the $H$ action on $X$ to a cocompact $H$-action on $\fhyp$ by Theorem~\ref{thm:actionPromotion}.  
        
        We define a discrete subset $\cP \subset \fhyp$ stabilized by $H$ to which we may apply Lemma~\ref{lemma:discrete_geometric}. 
        Let $v \in X^{(0)}$ be a vertex, and let $H_v = \Stab_H(v)$. The $H_v$-orbit of $f(v)$ in $\fhyp$ is a bounded set since the map $f$ quasi-conjugates the $H$-action on $X$ to the $H$-action on $\fhyp$. As in Proposition~\ref{prop:newModelSpace2}, the convex hull of $H_v\cdot f(v)$ has a center $v'$ that is invariant under the action of $H_v$. We claim that $H \cdot v'$ is a discrete subset of $\fhyp$.  Since $\Stab_H(v) \subset \Stab_H(v')$, if $h \cdot v' \neq v'$, then $h \cdot v \neq v$. If $d_{\fhyp}(v', h\cdot v')< \epsilon$, then there exists $C = C(\epsilon)$ depending on the quasi-conjugacy constants so that $d_X(v, h\cdot v) <C$. The simplicial complex $X$ is locally finite, so the set $\{h \, | \, d(v,h\cdot v)<C\}$ is finite.  Thus, the set $\{h \, | \, d(v', h\cdot v')<\epsilon\}$ is finite.
        
        The $H$-action on $\Hy_{\mathbb{F}}^n$ is given by a homomorphism $\Phi: H \rightarrow \Isom(\mathbb{H}_{\mathbb{F}}^n)$. 
        Let $\bar G = \Phi(G)$, $\bar G' = \Phi(G')$, and $\bar H = \Phi(H)$. 
        As $G$ and $G'$ act geometrically on $\mathbb{H}_{\mathbb{F}}^n$ we deduce that $\bar G = G / F$ and $\bar G' = G' / F'$, where $F$ and $F'$ are the finite kernels of the respective actions on $\mathbb{H}_{\mathbb{F}}^n$.
        Since $\bar H$ stabilizes $\mathcal{P}$, Lemma~\ref{lemma:discrete_geometric} implies that $\bar H$ is a uniform lattice so $\bar G$ and $\bar G'$ are finite-index subgroups of $\bar H$, so are commensurable in $\bar H$.
        Let $\hat G = \Phi^{-1}(\bar G \cap \bar G') \cap G$ and $\hat{G}' = \Phi^{-1}(\bar G \cap \bar G') \cap G'$.
        Then $\hat G / F \cong \bar G \cap \bar G' \cong \hat G' / F'$, and we conclude that $G$ and $G'$ are virtually isomorphic.
        If $G$ and $G'$ are residually finite, then we can assume that $F$ and $F'$ are trivial by first passing to finite-index subgroups that do not contain the non-trivial elements of $F$ and $F'$ respectively.
        In which case the argument implies that $G$ and $G'$ are abstractly commensurable.
        \end{proof}

    \section{Pairwise Amalgamations}

    We prove Theorem~\ref{thm_intro_surface_case} and Theorem~\ref{thm_intro_amalgamation_case} in this section. We use the following notation throughout. 
    
    \begin{nota}
     Let $G \cong \Sigma_1 * \Sigma_2$ and $G' \cong \Si_1' *\Si_2'$, where $\Sigma_i \cong \pi_1(M_i)$ and $\Sigma_i' \cong \pi_1(M_i')$ are fundamental groups of closed orientable hyperbolic manifolds. Suppose $G$ and $G'$ have a common model geometry. 
     By Proposition~\ref{prop:newModelSpace2}, the groups $G$ and $G'$ act geometrically on an ideal model geometry $Y$.
     Let $\fhyp \equiv Y_v \subset Y$ be a one-ended vertex space. 
     Let $I_v$ be the set of edges incident to $v$ in $T$, and let $y_e = Y_e \cap Y_v$ be the point of intersection between the respective edge and vertex space.
     Then $\mathcal{P}_v = \{y_e \mid e \in I_v\}$ is a discrete subset of points in $Y_v$ that coarsely covers $Y_v$.
     Let $H = \Isom(Y)$. Then, $Stab_H(Y_v) \leq H$, the subgroup of isometries of $Y$ that stabilize $Y_v$, stabilizes the set $\mathcal{P}_v$. 
    \end{nota}    
    
    \begin{lemma} \label{lemma:geo_action}
     The group $\Stab_H(Y_v) / \Fix_H(Y_v)$ acts on $Y_v$ geometrically. 
    \end{lemma}
    \begin{proof}
        Since the group $\Sigma_i$ acts cocompactly on $Y_v$,  the lemma follows from Lemma~\ref{lemma:discrete_geometric}.
    \end{proof}    

    Without loss of generality, suppose in the following lemmas that the subgroups $\Si_1 \leq G$ and $\Si_1' \leq G'$ stabilize the vertex space $Y_v$. 
    We briefly introduce some terminology: 
    We will say that a $G$ action on a tree is \emph{reduced} if the action is minimal and for each valence two vertex $v$ in $G \backslash T$, if the corresponding edges are distinct, then at least one of the associated edge groups properly embeds in the vertex group associated to $v$.
    A reduced $G$-tree can be obtained from a $G$ tree by taking a $G$-minimal subtree and then removing bad valence two vertices, so the pair of incident edges become a single edge. 
    We note that our notion of reduced is distinct from the notion in~\cite{bestvinafeighn91}, and is tailored to our particular needs.
    
    \begin{lemma} \label{lem:peripheralTransitivity}
     The groups $\Si_1$ and $\Si_1'$ act transitively on the set $\mathcal{P}_v$. 
    \end{lemma}
    \begin{proof}
     The lemma follows from the uniqueness of the reduced Stallings-Dunwoody decomposition of $G$ and $G'$ in this case.
     The quotient $G \backslash T$ contains two vertices $u, v$, corresponding to the cosets of $\Si_1$ and $\Si_2$, and a single edge.
     Indeed, if any other vertex existed in $G \backslash T$, then its associated vertex group is trivial.
     If $G \backslash T$ contains more than one edge then either: the graph is not simply connected, the graph contains spurs that are not $u, v$, or the graph is a subdivided edge connecting $u$ and $v$.
     In the first case, if $G \backslash T$ is not simply connected, there would exist a non-trivial homomorphism from $G$ to $\mathbb{Z}$ such that $\Si_1$ and $\Si_2$ were in the kernel, contradicting the fact that $G = \Si_1 * \Si_2$.
     If $G \backslash T$ has spurs that are not either $u$ or $v$, then the action of $G$ on $T$ is not minimal.
     Finally, if $G \backslash T$ is a subdivided edge joining $u$ to $v$, then since the vertices and edges inside the subdivided edge have trivial groups associated to them, the action of $G$ on $T$ would not be reduced.
     Thus, the points in $ \mathcal{P}_v$ correspond to edges in the same $G$-orbit, and indeed in the same $G_v$-orbit.
    \end{proof}

    \begin{lemma} \label{lemma:same_degree}
     There exists $d \in \N$ so that $\Si_1$ and $\Si_1'$ are index-$d$ subgroups of $\Stab_H(Y_v) / \Fix_H(Y_v)$.
     Moreover, if $Y_v \cong \mathbb{H}^2$ then $\Sigma_1 \cong \Si_1'$. More generally, if $Y_v \cong \fhyp$, then $M_1$ and $M_1'$ have the same volume.
    \end{lemma}
    \begin{proof}
     By Lemma~\ref{lemma:geo_action}, the action of $\Stab_H(Y_v) / \Fix_H(Y_v)$ on $Y_v$ is geometric with quotient $\cO_v$, a compact hyperbolic orbifold.
     Since $\Sigma_1$ and $\Sigma_1'$ act freely on $Y_v$ they embed in $\Stab_H(Y_v) / \Fix_H(Y_v)$, so we obtain finite-sheeted orbifold covering maps ${f: M_1 \rightarrow \cO_v}$ and ${f': M_1' \rightarrow \cO_v}$.
     It suffices to show that $f$ and $f'$ both have the same degree.
     By Lemma~\ref{lem:peripheralTransitivity} the groups $\Sigma_1$, $\Sigma_2$, and $\Stab_H(Y_v) / \Fix_H(Y_v)$ act transitively on $\cP_v$, so there exists points $m \in M_1$, $m' \in M_1'$, and $o \in \cO_v$ corresponding to the quotient of that orbit.
     The degrees of $f$ and $f'$ are determined by the local degrees of the covering at $m$ and $m'$, which can be read off from the orbifold data at $o$, specifically the order of the finite group associated to a chart corresponding to $o$.
     
     Thus, the Euler characteristic $\chi(M_1) = \chi(M_1') = d \chi(\cO_v)$. So if $M_1$ and $M_2$ are surfaces, then they are homeomorphic, hence $\Sigma_1 \cong \Si_1'$. 
     Otherwise we can deduce that volume of $M_1$ and $M_1'$ is simply $d$ times the volume of $\cO_v$.
    \end{proof}

    \begin{proof}[Proof of Theorem~\ref{thm_intro_surface_case}]
     As above, we can assume that $\Sigma_1$ and $\Sigma_1'$ stabilize a vertex space $Y_v$ and conclude as in Lemma~\ref{lemma:same_degree} that $\Sigma_1 \cong \Sigma_1'$.
     Then we can also deduce that $\Sigma_2$ and $\Sigma_2'$ stabilize a vertex space $Y_u$ and similarly conclude that  $\Sigma_2 \cong \Sigma_2'$
    \end{proof}
    
    \begin{proof}[Proof of Theorem~\ref{thm_intro_amalgamation_case}]
     As above, we can assume that $\Sigma_1$ and $\Sigma_1'$ stabilize a vertex space $Y_v$ and conclude as in Lemma~\ref{lemma:same_degree} that $M_1$ and $M_1'$ have the same volume.
     Then we can also deduce that $\Sigma_2$ and $\Sigma_2'$ stabilize a vertex space $Y_u$ and similarly conclude that $M_2$ and $M_2'$ also have the same volume.
    \end{proof}

   \section{Action Rigidity and Leighton's Theorem} \label{sec:commonCover}
   
   The goal of this section is to prove the following theorem. 
   
   \begin{thm} \label{thm_FINAL}
    Let $G$ be a finitely generated, infinite-ended group.
    Suppose that the Stallings-Dunwoody decomposition of $G$ contains at least one one-ended vertex group, and that all one-ended vertex groups are quasi-isometric to a rank-1 symmetric space.
    If $G$ and $G'$ share a common model geometry $X$, then there exists a quasi-isometry $f : X \rightarrow Y$ to an ideal model geometry $Y$ that quasi-conjugates the $ \Isom(X)$-action on $X$ to an isometric action of $\Isom(X)$ on $Y$.
    Let $F: \Isom(X) \rightarrow \Isom(Y)$ denote the induced homomorphism.
    If $G$ and $G'$ are both residually finite, then $F(G)$ and $F(G')$ are weakly commensurable in $\Isom(Y)$. 
   \end{thm}

   Abstract commensurability of $G$ and $G'$ follows from their residual finiteness, and the weak commensurability of $F(G)$ and $F(G')$. So Theorem~\ref{thm_first} follows immediately from the statement of Theorem~\ref{thm_FINAL}. 
   All that remains to be proven of Theorem~\ref{thm_FINAL} is given by the following statement: 
   
   \begin{thm} \label{thm:weakCommensurabilityInIdeal}
   Let $X$ be an ideal model geometry.
   Let $\Gamma$ and $\Gamma'$ be uniform lattices of $\Isom(X)$. 
   If $\Gamma$ and $\Gamma'$ are residually finite, then they are weakly commensurable in $\Isom(X)$.
   \end{thm}  
   
   \begin{proof}[Proof of Theorem~\ref{thm_FINAL}]
    Proposition~\ref{prop:newModelSpace2} yields the quasi-conjugacy~$f$.
    The statement follows immediately by an application of Theorem~\ref{thm:weakCommensurabilityInIdeal}, since the groups $F(G)$ and $F(G')$ will be quotients of $G$ and $G'$ by finite, normal subgroups, and will therefore also be residually finite.
   \end{proof}

   We take a moment to motivate the proof of Theorem~\ref{thm:weakCommensurabilityInIdeal}. See Figure~\ref{fig:cartoon} for an illustration. 
   Let $X$ be an ideal model geometry for $\Gamma, \Gamma' \leqslant \Isom(X)$.
   By residual finiteness, we can assume that $\Gamma$ and $\Gamma'$ are torsion free by passing to finite-index subgroups.
   The spaces $\chi = \Gamma \backslash X$ and $\chi' = \Gamma' \backslash X$ decompose as finite graphs of spaces with vertex spaces that are closed hyperbolic manifolds or points and edge spaces that are isometric to $[0,1]$. 
   We think of these spaces as being a hybrid between graphs and hyperbolic manifolds.
   The ultimate goal is to construct homeomorphic finite covers of $\chi$ and $\chi'$, which would imply their fundamental groups are abstractly commensurable.
   To construct these covers, we set up the framework to apply symmetry-restricted Leighton's theorem (Theorem~\ref{thm:symmetryRestrictedLeighton} in Section~\ref{sec:SymmetryRestrictedLeigthon}). 
   Importantly, this theorem applies only to locally finite trees and to groups that act freely and cocompactly on such a tree. 
    So, we find a common (infinite-sheeted) cover $\breve{\chi}$ of $\chi$ and $\chi'$ so that the underlying tree is locally finite and so that $\pi_1(\chi)$ and $\pi_1(\chi')$ virtually act freely on this tree by deck transformations. 

   In the first stage of our argument, we pass to finite covers $\hat{\chi} \rightarrow \chi$ and $\hat{\chi}' \rightarrow \chi'$ that are {\it locally isomorphic} in the sense that if $\hat \chi_v$ and $\hat \chi_v'$ are vertex spaces of $\hat \chi$ and $\hat \chi'$ that have lifts to $X$ in the same $\Isom(X)$ orbit, then $\hat \chi_v$ and $\hat \chi_v'$ are isometric.
   In Section~\ref{sec_step1} we construct common covers of the vertex spaces by taking a certain kind of normal core of the vertex groups. 
   In Section~\ref{sec_step2} we obtain $\hat{\chi}$ and $\hat{\chi}'$ by constructing quotient homomorphisms from $\pi_1(\chi)$ and $\pi_1(\chi')$ to virtually free groups obtained by quotienting the vertex groups by the normal cores obtained in the previous section.
   By passing to torsion-free finite-index subgroups of the virtually free groups, we obtain finite-index subgroups corresponding to $\hat{\chi}$ and $\hat{\chi}'$.
   In Section~\ref{sec_step3}, we prove that $\hat{\chi}$ and $\hat{\chi}'$ have isometric regular covers $\breve{\chi} \cong \breve{\chi}'$ which decompose as locally finite trees of spaces. 
   Finally, in Section~\ref{sec_step4}, we prove that if $\breve{T}$ is the underlying tree of the space $\breve{X}$, then $\Isom(\breve{X}) \leq \Isom(\breve{T})$ is symmetry restricted. 
   The images of $\pi_1(\hat{\chi})$ and $\pi_1(\hat{\chi}')$ in $\Isom(\breve{\chi})$ are free uniform lattices, so the main theorem will follow from an application of Theorem~\ref{thm:symmetryRestrictedLeighton}.

   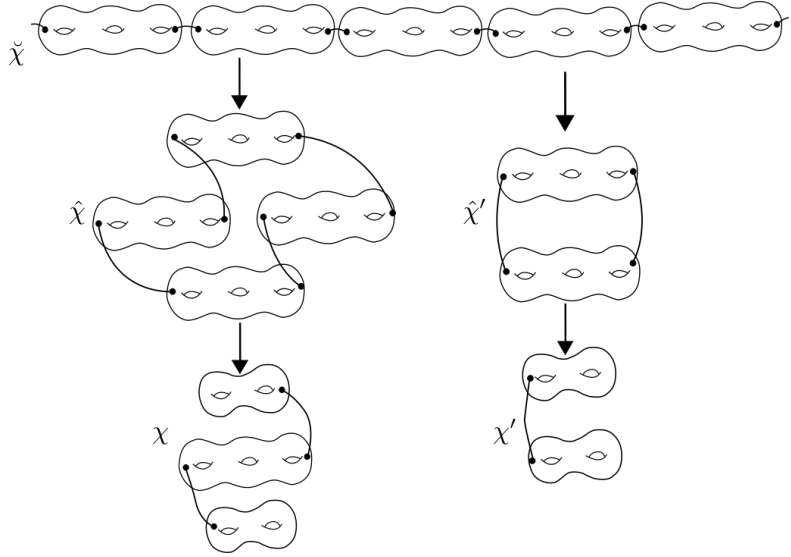
\begin{figure}
	\begin{overpic}[width=.65\textwidth,tics=5,]{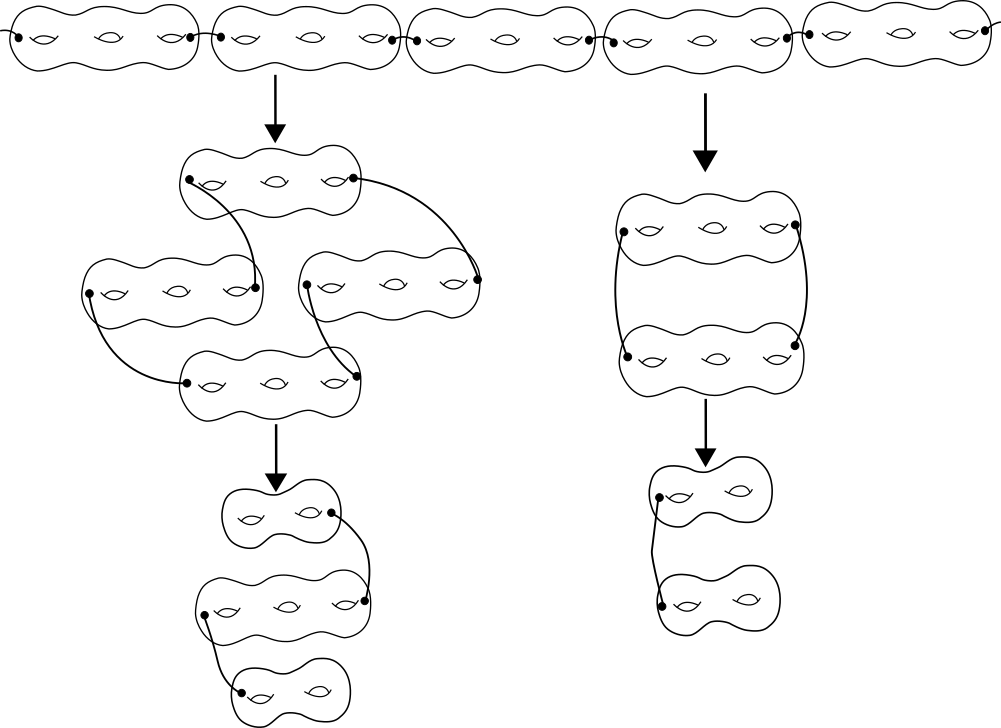} 
	\put(16,15){$\chi$}
	\put(61,15){$\chi'$}
        \put(5,44){$\hat{\chi}$}
        \put(57,44){$\hat{\chi}'$}
        \put(-3,65){$\breve{\chi}$}
          \end{overpic}
	\caption{An illustration of the proof up to the point of Proposition~\ref{prop:stiffUniversalCover}. The space $\breve{X}$ is a common cover of $\chi$ and $\chi'$ with locally finite underlying tree. 
	Note that the underlying tree of $\breve{\chi}$ in this example is the simplicial line, but in general it could be any locally finite tree.}
	\label{fig:cartoon}
    \end{figure}

\subsection{Finite covers of the vertex spaces} \label{sec_step1}
   
   Let $H = \Isom(X)$.  
   Let $v \in VT$, let $X_v^+ = N_1(X_v)$, let $H_v = Stab_H(X_v^+)$,  and let $$K_v = H_v / \textrm{Fix}_H(X_v^+) \leqslant \Isom(X_v^+).$$  
   Let $q_v: H_v \rightarrow K_v$ denote the quotient map.
   The space $X_v^+$ is either isomorphic to a vertex with a finite number of edges attached in a star, or it is isomorphic to $\mathbb{H}_{\mathbb{F}}^n$ with edges isometric to $[0,1]$ attached at a discrete subset of distinct points. 
   We consider the latter case where $X_v$ is isometric to $\mathbb{H}_{\mathbb{F}}^n$.
   We will refer to such $v \in VT$ as \emph{one-ended} vertices and denote their subset by $V_1T \subseteq VT$.

   The group $K_v$ acts geometrically on $X_v^+$. Indeed, since the edges attached to $X_v$ in $X_v^+$ are preserved by $K_v$, we deduce, as before (see the proof of Lemma~\ref{lemma:geo_action}), that $K_v$ acts properly on $X_v$. 
    Both $\Gamma_v$ and $\Gamma'_v$ are embedded in $K_v$ by the map $q_v$ (again, see the proof of Lemma~\ref{lemma:geo_action}). 
    So, $K_v$ acts cocompactly on $X_v$ and $\Gamma_v$ and $\Gamma_v'$ are embedded as finite-index subgroups of $K_v$ by Lemma~\ref{lemma:subgroup}.  
    Thus, the group $K_v$ acts geometrically on $X_v^+$.
    Moreover, for all $h \in H$ such $hu = v$ the subgroups $\Gamma_u^h$ and $(\Gamma'_u)^h$ are also embedded as finite-index subgroups of $K_v$ via the map $q_v$. 
   Since there are only finitely many $\Gamma$ and $\Gamma'$ orbits of vertices, there is a global upper bound on the index of $q_v(\Gamma_u^h), q_v((\Gamma_u')^h) \leqslant K_v$.
   Therefore, the following group is a finite-index normal subgroup of $K_v$
   \[
    \hat{K}_v := \bigcap_{\{h \in H \,|\, hu = v\}} q_v(\Gamma_u^h) \cap q_v\left((\Gamma_u')^h\right) \trianglelefteq K_v.
   \]   
   \noindent Indeed, to verify normality, let $k \in K_v$. The element $k$ is represented by some $h' \in H_v$.
   Then, 
   \begin{align*}
    \hat{K}_v^k  & = \left( \bigcap_{\{h \in H \,|\, hu = v\}} \Big(q_v(\Gamma_u^h) \cap q_v\left((\Gamma_u')^h\right)\Big)^k \right) \\
                 & = \left( \bigcap_{\{h'h \in H \,|\, h'hu = v\}} q_v(\Gamma_u^{h'h}) \cap q_v\left((\Gamma_u')^{h'h}\right) \right)\\
                 & = \hat{K}_v
   \end{align*}
   Moreover, since this same computation holds for all $k \in q_v(\Gamma_v)$ and $k \in q_v(\Gamma_v')$ we deduce that $\hat{K}_v$ is also a normal subgroup of $q_v(\Gamma_v)$ and $q_v(\Gamma_v')$. 
   
   If $v \in V_1T$, then $\hat{K}_v$ is a finite-index subgroup of both $\Gamma_v$ and $\Gamma_v'$. Let $\phi_v : \hat{\chi}_v \rightarrow \chi_v$ and $\phi_v' : \hat{\chi}_v' \rightarrow \chi_v'$ be the associated finite-sheeted regular covers.  

\subsection{Locally-isomorphic finite covers $\hat{\chi} \rightarrow \chi$ and $\hat{\chi}' \rightarrow \chi'$} \label{sec_step2}

   Suppose the spaces $\chi = \Gamma \backslash X$ and $\chi' = \Gamma' \backslash X$ have underlying graphs $\Upsilon$ and $\Upsilon'$, respectively.
   Recall, the \emph{cone} of a topological space $Z$ is defined to be the quotient space $\textrm{Cone}(Z) = \left( Z \times [0,1] \right) / \left( Z \times \{ 1 \} \right)$. 
   Define quotient spaces 
   \[Y = \left(\chi \bigsqcup_{u \in V_1\Upsilon} \textrm{Cone}(\hat{\chi}_u)\right) \Big/ \sim 
     \; \,\, \textrm{  and  } \; \;\,
     Y' = \left(\chi' \bigsqcup_{u \in V_1\Upsilon'} \textrm{Cone}(\hat{\chi}_u') \right) \Big/ \sim 
   \]
   where the equivalence relations $\sim$ are given by $\phi_u(x) \sim (x,0)$ and $\phi_u'(x') \sim (x',0)$ for each $u \in V_1\Upsilon$, where $\phi_u$ and $\phi_u'$ are the covering maps defined above.
   There are natural embeddings $\theta: \chi \rightarrow Y$ and $\theta' : \chi' \rightarrow Y'$.
   
   On a group theoretic level, 
   $$\pi_1 (\chi) = \pi_1(\chi_{v_1}) * \cdots * \pi_1(\chi_{v_m}) * F_\ell, $$
   where $v_1, \ldots, v_m$ are the one ended vertices in $\Upsilon$.
   Then, take the quotient $$\pi_1( Y) = \pi_1(\chi_{v_1}) / \pi_1(\hat{\chi}_{v_1}) * \cdots * \pi_1(\chi_{v_m}) / \pi_1(\hat{\chi}_{v_m}) * F_\ell.$$
   An analogous construction is applied to $\pi_1(\chi')$ to obtain $\pi_1(Y')$.
   The groups $\pi_1 (\chi_{v_i})/\pi_1(\hat{\chi}_{v_i})$ and $\pi_1 (\chi_{v_i}')/ \pi_1(\hat{\chi}_{v_i}')$ are finite groups. Hence, the groups $\pi_1(Y)$ and $\pi_1(Y')$ are virtually free. Thus, there exist finite-sheeted covers $\hat{Y} \rightarrow Y$ and $\hat{Y}' \rightarrow Y'$ with free fundamental groups.
   Let $\hat{\chi} \rightarrow \chi$ be the finite-sheeted cover corresponding to $\theta^{-1}_*(\pi_1 (\hat{Y}))$, and $\hat{\chi}' \rightarrow \chi'$ be the finite-sheeted cover corresponding to $(\theta')^{-1}_*(\pi_1( \hat{Y}'))$.
   The vertex spaces in $\hat{\chi}$ covering $\chi_u$ are isomorphic to $\hat{\chi}_u$ and the covering maps are precisely $\phi_u$.
   An analogous statement holds for the vertex spaces in $\hat{\chi}'$.

\subsection{A common regular cover with locally finite underlying graph} \label{sec_step3}

   There is a normal subgroup of $\pi_1 (\hat{\chi})$ generated by the vertex groups and all  their conjugates in $\pi_1( \hat{\chi})$.
   The corresponding regular cover $\breve{\chi} \rightarrow \hat{\chi}$ decomposes as a tree of spaces.
   The induced covering map $\breve{\chi}_v \rightarrow \hat{\chi}_u$, given by restricting to a vertex space, is an isometry.
   Alternatively, if $\hat{\Upsilon}$ is the underlying graph for $\hat{\chi}$, then $\breve{\chi}$ is the covering space determined by the universal cover of $\hat{\Upsilon}$.
   Similarly, we obtain the corresponding regular covering $\breve{\chi}' \rightarrow \hat{\chi}'$.
   
   \begin{prop} \label{prop:stiffUniversalCover}
    There is an isometry $\varphi : \breve{\chi} \rightarrow \breve{\chi}'$.
   \end{prop}

   \begin{proof}
    Let $\breve{T} := \pi_1 (\breve{\chi}) \backslash T$ and $\breve{T}' := \pi_1 (\breve{\chi}') \backslash T$ denote the underlying trees of $\breve{\chi}$ and $\breve{\chi}'$ respectively.
    The induced covering maps $P: X \rightarrow \breve \chi$ and $P' : X \rightarrow \breve \chi'$ respect the tree of spaces decomposition. So, there are quotient maps
    $p : T \rightarrow \breve{T}$ and $p' : T \rightarrow \breve{T}'$ so that $P$ induces a cover $X_v \rightarrow \breve \chi_{p(v)}$ for each $v\in VT$, and similarly for $P'$.
    Choose an exhaustive enumeration $u_0, u_1, u_2, \ldots$ of the vertices of $\breve{T}$ such that $\{u_0, \ldots, u_n\}$ span a subtree of $\breve{T}$ for all $n \in \N$.
    Let $\breve{\chi}_v^+$ and $\breve{\chi}_v'^+$ denote the $1$-neighborhood of the respective vertex spaces as before. 
    We will inductively define $\varphi_*: \breve{T} \rightarrow \breve{T}'$ and the map $\varphi$ along with it.
    
    Choose $v_0 \in V_1T$ such that $p(v_0) = u_0$.
    The covering map $X_{v_0}^+ \rightarrow \breve{\chi}_{u_0}^+$ is induced by quotienting by $\hat K_{v_0}$.
    Let $u_0' = p'(v_0)$ and observe that $X_{v_0}^+ \rightarrow \breve{\chi'}_{u_0'}^+$ is also obtained by quotienting by $\hat K_{v_0}$. 
    If $v_0 \in V_1T$ then $\breve{\chi}_{u_0}^+ = \hat{K}_{v_0} \backslash X_{v_0}^+ = \breve{\chi}_{u_0'}'^+$ so there is an isometry $\varphi_0: \breve{\chi}_{u_0}^+ \rightarrow \breve{\chi}_{u_0'}'^+$ such that the following diagram commutes:
    \[
     \xymatrix{
      & X_{v_0}^+ \ar[dl] \ar[dr] & \\
      \breve{\chi}_{u_0}^+ \ar[rr]^{\varphi_0} & & \breve{\chi}_{u_0'}'^+ \\
     }
    \]
   
    Now, we may proceed inductively, assuming that isometries $\varphi_0, \ldots, \varphi_{n-1}$ have been defined, and the map $\varphi_*$ is defined on all vertices $u_0, \ldots, u_{n-1}$ and their incident edges.
    For all $i < n$ there exists $v_i \in VT$ and a  map $\varphi_i$ has been defined such that the following commutes:
    \[
     \xymatrix{
     & X_{v_i}^+ \ar[dl] \ar[dr] & \\
      \breve{\chi}_{u_i}^+ \ar[rr]^{\varphi_i} & & \breve{\chi}_{u_i'}'^+ \\
     }
    \]
    
     The vertex $u_n$ is incident to some $u_j$ for $j < n$ via an edge $e_n$.
     Then there exists a vertex $v_n \in VT$  incident to $v_j$ via an edge $\tilde e_n$ such that $p(v_n) = u_n$ and $p(\tilde e_n) = e_n$.
     Let $u_n' = p'(v_n)$ and $e_n' = p'(\tilde e_n)$.
     Note that $e_n'$ connects $u_j'$ to $u_n'$.
    If $v_n \in V_1T$ then as in the initial case $\breve{\chi}_{u_n}^+ = \hat{K}_{v_n} \backslash X_{v_n}^+ = \breve{\chi}_{u_n'}'^+$ so there is an isometry $\varphi_n: \breve{\chi}_{u_n}^+ \rightarrow \breve{\chi}_{u_n'}'^+$ such that the following diagram commutes:
    \[
     \xymatrix{
      & X_{v_n}^+ \ar[dl] \ar[dr] & \\
      \breve{\chi}_{u_n}^+ \ar[rr]^{\varphi_n} & & \breve{\chi}_{u_n'}'^+ \\
     }
    \]
    Otherwise, if $v_n \notin V_1T$ then  the covering maps $X_{v_n}^+ \rightarrow \breve{\chi}_{u_n}^+$ and  $X_{v_n}^+ \rightarrow \breve{\chi}_{u_n'}'^+$ are isometries and $\varphi_n$ is given by the composition $\breve{\chi}_{u_n'} \rightarrow X_{v_n}^+ \rightarrow \breve{\chi}_{u_n'}'^+$ so the diagram commutes.
    Again, as in the initial case, if $e$ is an edge incident to $u_n$, then $\varphi_n(\breve \chi_e) = \breve \chi_{e'}$, where $e$ and $e'$ both correspond to the same $\hat{K}_{v_n}$-orbit of edge incident to $v_n$ in $T$.
    The edge space $X_{\tilde e_n}$ covers $\breve \chi_{e_n}$ and $\chi_{e_n'}'$ respectively by an isometry, so $\varphi_n$  will map $\breve \chi_{e_n}$ to $\breve \chi_{e_n'}$
    so we deduce that $\varphi_n$ is consistent with $\varphi_j$ on the edge space $\breve{\chi}_{e_n} = \breve{\chi}_{u_j}^+ \cap \breve{\chi}_{u_n}^+$.
    
    The induction is complete and taking all the $\varphi_n$ gives a well defined function $\varphi$ such that $\varphi_*$ is a local isometry between two trees, and hence an isomorphism. 
   \end{proof}
 
    Identifying $\breve{\chi}$ with $\breve{\chi}'$, we can say that $\breve{\chi}$ is a common regular cover of both $\hat{\chi}$ and $\hat{\chi}'$.

\subsection{The group $\Isom(\breve{\chi})\leq \Isom(\breve{T})$ is symmetry restricted}   \label{sec_step4}  
    
    Recall the notion of a {\it symmetry restricted subgroup} given in Definition~\ref{def:symmres}. 
    There are maps $P: \pi_1(\hat{\chi}) \rightarrow \Isom(\breve{\chi})$ and $P' : \pi_1 (\hat{\chi}') \rightarrow \Isom(\breve{\chi})$.
     Let $\Phi: \Isom(\breve{\chi}) \rightarrow \Isom(\breve{T})$ be the natural map induced by the tree of spaces decomposition $\breve{X} \rightarrow \breve{T}$. 
     Let $F := \Phi \circ P(\pi_1( \hat{\chi}))$ and $F' := \Phi \circ P'(\pi_1 (\hat{\chi}'))$.

    \begin{lemma} \label{lemma_free_uniform_breve}
     The groups $F$ and $F'$ are free uniform lattices in $\Aut(\breve{T})$.
    \end{lemma}
    \begin{proof}
      The fundamental groups $\pi_1(\hat{\chi})$ and $\pi_1(\hat{\chi}')$ act cocompactly on $\breve{\chi}$ by construction. 
      To show that $F$ and $F'$ are free, we claim that $F$ and $F'$ act freely on the locally finite tree $\breve{T}$. Indeed, if a vertex $v \in \breve{T}$ is stabilized by an element $\Phi \circ P(g)$, then $g$ must stabilize $\breve{\chi}_v$, and so fixes $\breve{\chi}_v$, since by construction the covering map $\breve{\chi}_v \rightarrow \hat{\chi}_v$ is an isometry.
    \end{proof}

    \begin{lemma} \label{lemma:symmetry_restricted}
       The group $\breve{H}:=\Phi\bigl(\Aut(\breve{\chi})\bigr)$ is a $1$-symmetry restricted subgroup of $\Aut(\breve{T})$; that is, $${\breve{H} = \mathscr{S}_1(\breve{H}) \leqslant \Aut(\breve{T})}.$$ 
    \end{lemma}

    \begin{proof}
     Suppose that $g \in \mathscr{S}_1(\breve{H})$. We wish to find $h \in \Aut(\breve{\chi})$ such that $\Phi(h) = g$. By the definition of $\mathscr{S}_1(\breve{H})$, for each $v \in V\breve{T}$ the restriction $g_v : N_1(v) \rightarrow N_1(g v)$ is equal to the restriction of some $h_v \in \Aut(\breve{X})$ to $N_1(\breve{\chi}_v) = \breve{\chi}_v^+$. 
     If $u$ and $v$ are adjacent vertices then the isometries $h_u$ and~$h_v$ agree on the edge space $\breve{\chi}_e = \breve{\chi}_u^+ \cap \breve{\chi}_v^+$, where $e$ is the edge connecting $u$ to $v$. See Figure~\ref{fig:symm_res}.
     Thus, we can define an isometry $h$ of $\breve{\chi}$ to be $h_v$ on~$\breve{\chi}_v^+$.
    \end{proof}

           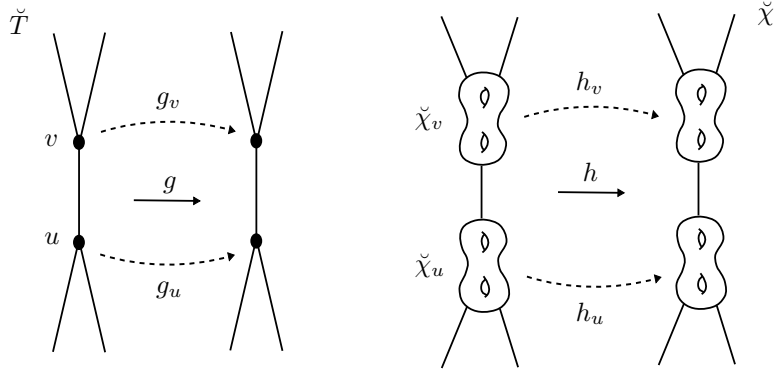
\begin{figure}
	\begin{overpic}[width=.6\textwidth, tics=5,]{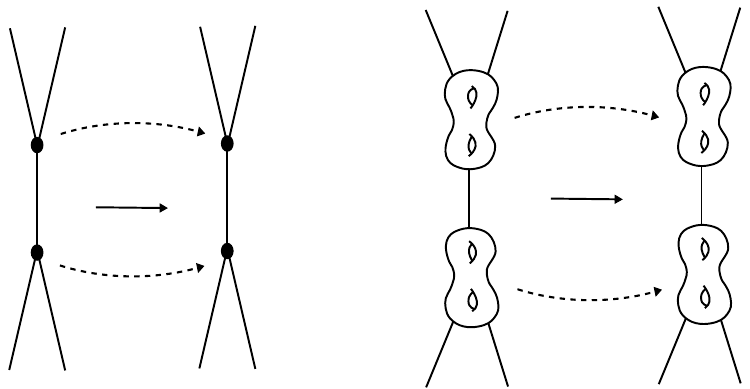} 
	    \put(0,32){$v$}
	    \put(0,18){$u$}
	    \put(17,26){$g$}
	    \put(16,38){$g_v$}
	    \put(16,11){$g_u$}
	    \put(76,40){$h_v$}
	    \put(77,27){$h$}
	    \put(76,7){$h_u$}
	    \put(-5,48){$\breve{T}$}
	    \put(102,50){$\breve{\chi}$}	  
	    \put(53,35){$\breve{\chi}_v$}
	    \put(53,14){$\breve{\chi}_u$}  
          \end{overpic}
	\caption{An illustration of the gluing in Lemma~\ref{lemma:symmetry_restricted}. }
	\label{fig:symm_res}
    \end{figure}

     \begin{proof}[Proof of Theorem~\ref{thm:weakCommensurabilityInIdeal}]
      By applying Theorem~\ref{thm:symmetryRestrictedLeighton} we obtain $\breve{h} \in \breve{H}$ such that $\bar{F} = \breve{h}F\breve{h}^{-1} \cap F'$ is a finite-index subgroup of both $\breve{h}F\breve{h}^{-1}$, and $F'$.
      Thus $\breve{h}^{-1} \bar{F} \breve{h}$ is a finite-index subgroup of $F$ with associated finite cover $\bar{\chi} \rightarrow \hat{\chi}$, that is isometric to the finite cover $\bar{\chi}' \rightarrow \hat{\chi}'$ associated to $\bar{F}$, a finite-index subgroup of $F'$.
      Thus we have an element $h \in H$ such that $(\pi_1 \bar \chi )^h$ is equal to $\pi_1 \bar \chi'$, and thus $h$ commensurates $\Gamma$ to $\Gamma'$ in $H$.
     \end{proof}

   \section{Free products of finite extensions} \label{sec:hypotheticalCounterexample}

   Let $G$ be a uniform lattice in $\Isom(\mathbb{H}^n_{\mathbb{F}})$, and suppose that there exists a finite extension 
   $$ 1 \rightarrow F \rightarrow E \rightarrow G \rightarrow 1$$
   such that $E$ is not residually finite.
   Since $G$ is residually finite, by taking the preimage in $E$ of a torsion-free, finite-index subgroup of $G$, we can assume that $G$ is torsion free.
   The non-residual finiteness of $E$ implies that there exists some element of $E$ that lives in every finite-index subgroup of~$E$.
   Since $G$ is residually finite, this element must lie in $F$.
   Let $f \in F$ denote such an element. As $F$ is finite, let $r$ be the order of $f$.
   
   Consider the hyperbolic triangle group 
   $$\Delta = \Delta(r,5,5) = \langle a, b, c \mid a^2 = b^2 = c^2 = (ab)^r = (bc)^5 = (ac)^5 =1 \rangle.
   $$
   Consider the following amalgamated free product:
   \[
    H := E *_{\langle f \rangle = \langle ab \rangle} \Delta
   \]
   We will show the following:
   
   \begin{prop}
    Under the assumption that $E$ is not residually finite, the group $H$ is not action rigid.
    Moreover, there exists $H'$ that shares a common model geometry with $H$, but is not even virtually isomorphic to~$H$.
   \end{prop}
   
   \begin{proof}
   First we construct a model geometry for $H$.
   Let $T$ be the Bass-Serre tree for $H$.
   Let $v$ be the vertex stabilized by $E \leqslant H$ and $u$ be the adjacent vertex stabilized by $\Delta \leqslant H$. Let $e=\{v,u\}$. 
   Let $X_v := \mathbb{H}_{\mathbb{F}}^n$.
   We have a geometric action of $E$ on $X_v$ given by the surjection onto $G$.
   Fix a basepoint $x_v \in X_v$. 
   Note that $x_v$ is fixed by $F$, while $G$ itself acts freely on $X_v$.
   Let $X_u := \mathbb{H}^2$ with the associated $\Delta$-action.
   Fix the basepoint $x_u \in X_u$ to be the unique fixed point of the torsion element $ab \in \Delta$.
   
   Let $1,h_1,h_2,\ldots$ be coset representatives of $H/H_v$ and let $[h]_v$ denote the coset containing $h$.
   The vertices $v, h_1v, h_2v,\ldots $ are the $H$-orbit of $v$.
   Similarly, let $1,h'_1,h'_2,\ldots$ be coset representatives of $H/H_u$ and $[h]_u$ denote the corresponding coset.
   The vertices $u, h_1u, h_2u,\ldots $ are the $H$-orbit of $u$.
   Define an action of $H$ on $X_v \times H/H_v$ by letting
   \[
    h \cdot (x, [h_i]_v) = (h_j^{-1} h h_i \cdot x, [h_j]_v = [hh_i]_v),
   \]
   which is well defined since $[h_j]_v =[hh_i]_v$ implies that $h_j^{-1}hh_i \in H_v$.
   Similarly, define an action of $H$ on $X_u \times H/H_u$ by letting
   \[
    h \cdot (x, [h_i']_u) = ({h'}_j^{-1} h h'_i \cdot x, [h'_j]_u = [hh'_i]_u).
   \]
   For $\hat h \in H$, let $[\hat h]_e$ denote the coset $\hat h H_e$ in $H/H_e$. Define the $H$-action on $[0,1] \times H / H_e$ by
   \[
     h \cdot (\alpha, [\hat h]_e) = (\alpha, [ h \hat h]_e)
   \]
   for $\hat h \in H$. We define the model geometry to be the quotient space 
   $$X := (X_v \times H/H_v) \times (X_u \times H/H_u) \times ([0,1] \times H/H_e) \slash \sim$$
   by letting $$(0, [\hat h]_e) \sim (h_i^{-1} \hat h x_v, [\hat h]_v = [h_i]_v)$$
   and 
   $$(1, [\hat h]_e) \sim ({h_i'}^{-1} \hat h x_u, [\hat h]_u = [h_i']_u).$$
   Note that this does not depend on the choice of representative $\hat h$ since $H_e$ fixes $x_u$ and $x_v$. 
   We can verify that the $H$ action on the disjoint union passes to the quotient space by observing that
   \begin{align*}
     h \cdot (0, [\hat h]_e) & = (0, [h \hat h]_e) \\
                             & \sim ( h_j^{-1} h \hat h x_v, [h \hat h]_v = [h_j]_v) \\
                             & = \big( (h_j^{-1}h h_i) (h_i^{-1} \hat h x_v) , [ h h_i]_v = [h_j]_v \big) \\
                             & = h \cdot (h_i^{-1} \hat h x_v, [h]_v = [h_i]_v).
   \end{align*}
  Since $(0, [\hat h]_e) \sim (h_i^{-1} \hat h x_v, [h]_v = [h_i]_v)$, we conclude that the action respects the relation $\sim$, so $H$ acts on $X$.

  The space $X$ decomposes as an $H$-equivariant tree of spaces via the map 
  $$ p: X \rightarrow T$$
  given by mapping $(x, [h_i]_v) \mapsto h_iv$ and $(x, [h_i']_u) \mapsto h'_i u$ and linearly mapping $[0,1] \times [\hat h]_e$ to the edge $\hat h e$ so that $(0, [\hat h]_e) \mapsto \hat h v$ and $(1, [\hat h]_e) \mapsto \hat h u$.
  
  The vertices of the underlying tree $T$ are bicoloured according to whether the vertex space is isometric to $\mathbb{H}^n_{\mathbb{F}}$ or $\mathbb{H}^2$.
  Within each vertex space isometric to $\mathbb{H}^n_{\mathbb{F}}$ there are $|F|/ r$ edge spaces attached to each $G$-orbit point of $x_v$, and in each vertex space isometric to $\mathbb{H}^2$ there is a single edge space attached to each $\Delta$-orbit point of $x_u$.
  Any other space that decomposes in the fashion described will be isometric to $X$, since a suitably chosen isometry between vertex spaces can be extended to the incident edge space, then the neighbouring vertex spaces, and so on through the entire tree of spaces.

   We now construct $H'$.
   Let $p = |F|/r$.
   Let $\Delta' \leqslant \Delta$ be a torsion-free, finite-index subgroup.
   Note that $r$ divides $[\Delta: \Delta']$ so let $q = [\Delta : \Delta'] / r$.
   Let 
   \[
    H' = (*_{i=1}^q G) * (*_{i=1}^p \Delta') *\mathbb{F}_{pq - p - q +1}
   \]
   Realize $H'$ as the fundamental group of a graph of spaces obtained by taking $q$ copies of $G \backslash X_v$ and $p$ copies of $\Delta' \backslash X_u$.
   The underlying graph will be the bipartite graph $\Gamma$ with $q$ vertices of valence $p$ with vertex space $G \backslash X_v$ and $p$-vertices of valence $q$ with vertex space  $\Delta' \backslash X_u$.
   The edge spaces will be intervals with endpoints attached to either the basepoint $\bar x_v \in G \backslash X_v$ covered by $x_v$, or the $q$ distinct points in $\Delta' \backslash X_u$ covered by the orbits of $\Delta x_u \subseteq X_u$.
   The resulting graph of spaces $Y$ will have universal cover isometric to $X$.   
   
   To see that $H$ and $H'$ are not virtually isomorphic, suppose $J$ is a finite normal subgroup of $\Upsilon$, where $\Upsilon$ is a finite-index subgroup of $H$ or $H'$. Then, $J$ will fix a vertex of $T$ (since $J$ is finite), and therefore the entire $\Upsilon$-orbit of that vertex (since $J$ is normal), and thus $J$ will fix the entire tree $T$.
   The only elements of $H$ or $H'$ that fix the entire tree are the identity, so $J$ must be trivial.
   Thus if $H$ and $H'$ are virtually isomorphic, then they must be abstractly commensurable, which is impossible since $H'$ is residually finite, $H$ is not, and residual finiteness is a commensurability invariant.
   \end{proof}


\bibliographystyle{alpha}
\bibliography{Ref}

\newcommand{\etalchar}[1]{$^{#1}$}
\begin{thebibliography}{GJZPZ14}

\bibitem[Ago13]{agol}
Ian Agol.
\newblock The virtual {H}aken conjecture.
\newblock {\em Doc. Math.}, 18:1045--1087, 2013.
\newblock With an appendix by Agol, Daniel Groves, and Jason Manning.

\bibitem[Bas93]{Bass93}
Hyman Bass.
\newblock Covering theory for graphs of groups.
\newblock {\em J. Pure Appl. Algebra}, 89(1-2):3--47, 1993.

\bibitem[Bau63]{Baumslag63}
Gilbert Baumslag.
\newblock On the residual finiteness of generalised free products of nilpotent
  groups.
\newblock {\em Trans. Amer. Math. Soc.}, 106:193--209, 1963.

\bibitem[BF91]{bestvinafeighn91}
Mladen Bestvina and Mark Feighn.
\newblock Bounding the complexity of simplicial group actions on trees.
\newblock {\em Invent. Math.}, 103(3):449--469, 1991.

\bibitem[BH99]{bridsonhaefliger}
Martin~R. Bridson and Andr{\'e} Haefliger.
\newblock {\em Metric spaces of non-positive curvature}, volume 319 of {\em
  Grundlehren der Mathematischen Wissenschaften [Fundamental Principles of
  Mathematical Sciences]}.
\newblock Springer-Verlag, Berlin, 1999.

\bibitem[BJN09]{behrstockjanuszkiewiczneumann}
Jason~A. Behrstock, Tadeusz Januszkiewicz, and Walter~D. Neumann.
\newblock Commensurability and {QI} classification of free products of finitely
  generated abelian groups.
\newblock {\em Proc. Amer. Math. Soc.}, 137(3):811--813, 2009.

\bibitem[BK90]{BassKulkarni90}
Hyman Bass and Ravi Kulkarni.
\newblock Uniform tree lattices.
\newblock {\em J. Amer. Math. Soc.}, 3(4):843--902, 1990.

\bibitem[BM00]{burgermozes}
Marc Burger and Shahar Mozes.
\newblock Lattices in product of trees.
\newblock {\em Inst. Hautes \'Etudes Sci. Publ. Math.}, (92):151--194 (2001),
  2000.

\bibitem[BN12]{behrstockneumann12}
Jason~A. Behrstock and Walter~D. Neumann.
\newblock Quasi-isometric classification of non-geometric 3-manifold groups.
\newblock {\em J. Reine Angew. Math.}, 669:101--120, 2012.

\bibitem[Bow06]{bowditch06}
Brian~H. Bowditch.
\newblock {\em A course on geometric group theory}, volume~16 of {\em MSJ
  Memoirs}.
\newblock Mathematical Society of Japan, Tokyo, 2006.

\bibitem[BT78]{BaumslagTretkoff78}
Benjamin Baumslag and Marvin Tretkoff.
\newblock Residually finite {HNN} extensions.
\newblock {\em Comm. Algebra}, 6(2):179--194, 1978.

\bibitem[BW12]{BergeronWise}
Nicolas Bergeron and Daniel~T. Wise.
\newblock A boundary criterion for cubulation.
\newblock {\em Amer. J. Math.}, 134(3):843--859, 2012.

\bibitem[CdlH16]{CornulierDeLaHarpe}
Yves Cornulier and Pierre de~la Harpe.
\newblock {\em Metric geometry of locally compact groups}, volume~25 of {\em
  EMS Tracts in Mathematics}.
\newblock European Mathematical Society (EMS), Z\"{u}rich, 2016.
\newblock Winner of the 2016 EMS Monograph Award.

\bibitem[Cho96]{Chow96}
Richard Chow.
\newblock Groups quasi-isometric to complex hyperbolic space.
\newblock {\em Trans. Amer. Math. Soc.}, 348(5):1757--1769, 1996.

\bibitem[CJ94]{CassonJungreis94}
Andrew Casson and Douglas Jungreis.
\newblock Convergence groups and {S}eifert fibered {$3$}-manifolds.
\newblock {\em Invent. Math.}, 118(3):441--456, 1994.

\bibitem[Del78]{Deligne78}
Pierre Deligne.
\newblock Extensions centrales non r\'{e}siduellement finies de groupes
  arithm\'{e}tiques.
\newblock {\em C. R. Acad. Sci. Paris S\'{e}r. A-B}, 287(4):A203--A208, 1978.

\bibitem[DK18]{DrutuKapovich}
Cornelia Dru\c{t}u and Michael Kapovich.
\newblock {\em Geometric group theory}, volume~63 of {\em American Mathematical
  Society Colloquium Publications}.
\newblock American Mathematical Society, Providence, RI, 2018.
\newblock With an appendix by Bogdan Nica.

\bibitem[DT16]{DasTessera16}
Kajal Das and Romain Tessera.
\newblock Integrable measure equivalence and the central extension of surface
  groups.
\newblock {\em Groups Geom. Dyn.}, 10(3):965--983, 2016.

\bibitem[Dun85]{dunwoody85}
M.~J. Dunwoody.
\newblock The accessibility of finitely presented groups.
\newblock {\em Invent. Math.}, 81(3):449--457, 1985.

\bibitem[FPP{\etalchar{+}}]{FriedlEtAl}
Stefan Friedl, JungHwan Park, Bram Petri, Jean Raimbault, and Arunima Ray.
\newblock Nonmanifold hyperbolic groups of high cohomological dimension.
\newblock https://arxiv.org/abs/1807.09861.

\bibitem[Gab]{Gaboriau2019}
Damien Gaboriau.
\newblock On the top-dimensional {$l^2$}-betti numbers.
\newblock https://arxiv.org/abs/1909.01633.

\bibitem[Gab92]{gabai92}
David Gabai.
\newblock Convergence groups are {F}uchsian groups.
\newblock {\em Ann. of Math. (2)}, 136(3):447--510, 1992.

\bibitem[Gab02]{Gaboriau02}
Damien Gaboriau.
\newblock Invariants {$l^2$} de relations d'\'{e}quivalence et de groupes.
\newblock {\em Publ. Math. Inst. Hautes \'{E}tudes Sci.}, (95):93--150, 2002.

\bibitem[GJZPZ14]{GrunewaldJaikin-ZapirainPintoZalesskii14}
Fritz Grunewald, Andrei Jaikin-Zapirain, Aline G.~S. Pinto, and Pavel~A.
  Zalesskii.
\newblock Normal subgroups of profinite groups of non-negative deficiency.
\newblock {\em J. Pure Appl. Algebra}, 218(5):804--828, 2014.

\bibitem[GJZZ08]{GrunewaldJaikin-ZapirainZalesskii08}
F.~Grunewald, A.~Jaikin-Zapirain, and P.~A. Zalesskii.
\newblock Cohomological goodness and the profinite completion of {B}ianchi
  groups.
\newblock {\em Duke Math. J.}, 144(1):53--72, 2008.

\bibitem[GPS88]{GromovPiatetski-Shapiro}
M.~Gromov and I.~Piatetski-Shapiro.
\newblock Nonarithmetic groups in {L}obachevsky spaces.
\newblock {\em Inst. Hautes \'{E}tudes Sci. Publ. Math.}, (66):93--103, 1988.

\bibitem[Gro87]{gromov}
M.~Gromov.
\newblock Hyperbolic groups.
\newblock In {\em Essays in group theory}, volume~8 of {\em Math. Sci. Res.
  Inst. Publ.}, pages 75--263. Springer, New York, 1987.

\bibitem[Hil19]{Hill19}
Richard~M. Hill.
\newblock Non-residually finite extensions of arithmetic groups.
\newblock {\em Res. Number Theory}, 5(1):Paper No. 2, 27, 2019.

\bibitem[Hin85]{hinkkanen85}
A.~Hinkkanen.
\newblock Uniformly quasisymmetric groups.
\newblock {\em Proc. London Math. Soc. (3)}, 51(2):318--338, 1985.

\bibitem[Hin90]{hinkkanen90}
A.~Hinkkanen.
\newblock The structure of certain quasisymmetric groups.
\newblock {\em Mem. Amer. Math. Soc.}, 83(422):iv+87, 1990.

\bibitem[HW09]{HruskaWise09}
G.~Christopher Hruska and Daniel~T. Wise.
\newblock Packing subgroups in relatively hyperbolic groups.
\newblock {\em Geom. Topol.}, 13(4):1945--1988, 2009.

\bibitem[KL01]{KleinerLeeb01}
Bruce Kleiner and Bernhard Leeb.
\newblock Groups quasi-isometric to symmetric spaces.
\newblock {\em Comm. Anal. Geom.}, 9(2):239--260, 2001.

\bibitem[KL09]{KleinerLeeb09}
Bruce Kleiner and Bernhard Leeb.
\newblock Induced quasi-actions: a remark.
\newblock {\em Proc. Amer. Math. Soc.}, 137(5):1561--1567, 2009.

\bibitem[Lei82]{Leighton}
Frank~Thomson Leighton.
\newblock Finite common coverings of graphs.
\newblock {\em J. Combin. Theory Ser. B}, 33(3):231--238, 1982.

\bibitem[Mar06]{markovic06}
Vladimir Markovic.
\newblock Quasisymmetric groups.
\newblock {\em J. Amer. Math. Soc.}, 19(3):673--715, 2006.

\bibitem[Mos73]{Mostow73}
G.~D. Mostow.
\newblock {\em Strong rigidity of locally symmetric spaces}.
\newblock Princeton University Press, Princeton, N.J.; University of Tokyo
  Press, Tokyo, 1973.
\newblock Annals of Mathematics Studies, No. 78.

\bibitem[MR03]{MaclachlanReid}
Colin Maclachlan and Alan~W. Reid.
\newblock {\em The arithmetic of hyperbolic 3-manifolds}, volume 219 of {\em
  Graduate Texts in Mathematics}.
\newblock Springer-Verlag, New York, 2003.

\bibitem[MS15]{martinswiatkowski}
Alexandre Martin and Jacek \'{S}wiatkowski.
\newblock Infinitely-ended hyperbolic groups with homeomorphic {G}romov
  boundaries.
\newblock {\em J. Group Theory}, 18(2):273--289, 2015.

\bibitem[MSW03]{moshersageevwhyteI}
Lee Mosher, Michah Sageev, and Kevin Whyte.
\newblock Quasi-actions on trees. {I}. {B}ounded valence.
\newblock {\em Ann. of Math. (2)}, 158(1):115--164, 2003.

\bibitem[NR92]{NeumannReid}
Walter~D. Neumann and Alan~W. Reid.
\newblock Arithmetic of hyperbolic manifolds.
\newblock In {\em Topology '90 ({C}olumbus, {OH}, 1990)}, volume~1 of {\em Ohio
  State Univ. Math. Res. Inst. Publ.}, pages 273--310. de Gruyter, Berlin,
  1992.

\bibitem[Pan89]{Pansu89}
Pierre Pansu.
\newblock M\'{e}triques de {C}arnot-{C}arath\'{e}odory et quasiisom\'{e}tries
  des espaces sym\'{e}triques de rang un.
\newblock {\em Ann. of Math. (2)}, 129(1):1--60, 1989.

\bibitem[PW02]{papasogluwhyte}
Panos Papasoglu and Kevin Whyte.
\newblock Quasi-isometries between groups with infinitely many ends.
\newblock {\em Comment. Math. Helv.}, 77(1):133--144, 2002.

\bibitem[Ser80]{serre}
Jean-Pierre Serre.
\newblock {\em Trees}.
\newblock Springer-Verlag, Berlin-New York, 1980.
\newblock Translated from the French by John Stillwell.

\bibitem[SGW]{ShepherdGardamWoodhouse}
Sam Shepherd, Giles Gardam, and Daniel~J. Woodhouse.
\newblock Two generalisations of {L}eighton's theorem.
\newblock https://arxiv.org/abs/1908.00830.

\bibitem[Sta68]{stallings}
John~R. Stallings.
\newblock On torsion-free groups with infinitely many ends.
\newblock {\em Ann. of Math. (2)}, 88:312--334, 1968.

\bibitem[SW79]{scottwall}
Peter Scott and Terry Wall.
\newblock Topological methods in group theory.
\newblock In {\em Homological group theory ({P}roc. {S}ympos., {D}urham,
  1977)}, volume~36 of {\em London Math. Soc. Lecture Note Ser.}, pages
  137--203. Cambridge Univ. Press, Cambridge-New York, 1979.

\bibitem[SW18]{starkwoodhouse}
Emily Stark and Daniel Woodhouse.
\newblock Quasi-isometric groups with no common model geometry.
\newblock {\em J. Lond. Math. Soc.}, 2018.
\newblock To appear.

\bibitem[Swe01]{swenson}
Eric~L. Swenson.
\newblock Quasi-convex groups of isometries of negatively curved spaces.
\newblock {\em Topology Appl.}, 110(1):119--129, 2001.
\newblock Geometric topology and geometric group theory (Milwaukee, WI, 1997).

\bibitem[Tre73]{Tretkoff73}
Marvin Tretkoff.
\newblock The residual finiteness of certain amalgamated free products.
\newblock {\em Math. Z.}, 132:179--182, 1973.

\bibitem[Tuk86]{tukia}
Pekka Tukia.
\newblock On quasiconformal groups.
\newblock {\em J. Analyse Math.}, 46:318--346, 1986.

\bibitem[Why99]{whyte}
Kevin Whyte.
\newblock Amenability, bi-{L}ipschitz equivalence, and the von {N}eumann
  conjecture.
\newblock {\em Duke Math. J.}, 99(1):93--112, 1999.

\bibitem[Wis]{wise_hier}
Dani Wise.
\newblock The structure of groups with a quasiconvex hierarchy.
\newblock http://comet.lehman.cuny.edu/behrstock/cbms/program.html, unpublished
  manuscript.

\bibitem[Woo]{Woodhouse18}
Daniel~J. Woodhouse.
\newblock Revisiting {L}eighton's theorem with the {H}aar measure.
\newblock Preprint at: arXiv 1806.08196.

\end{thebibliography}

\end{document}